\documentclass[11pt,a4paper]{article}

\RequirePackage{amsthm,amsmath,amsfonts,amssymb}
\RequirePackage[numbers]{natbib}
\RequirePackage[colorlinks,citecolor=blue,urlcolor=blue]{hyperref}
\RequirePackage{graphicx}
\graphicspath{ {figures/} }

\newtheorem{theorem}{Theorem}[section]

\theoremstyle{remark}

\begin{document}

\title{\bf Bivariate densities in Bayes spaces: orthogonal decomposition and spline representation}

\author{Karel Hron$^{1}$, Jitka Machalov\'a$^{1}$ and Alessandra Menafoglio$^{2*}$ \\\\
    \small $^1$Department of Mathematical Analysis and Applications of Mathematics,\\
    \small Faculty of Science, Palack\'y University, Olomouc, Czech Republic,\\
    \small $^2$MOX, Department of Mathematics, Politecnico di Milano, Milano, Italy\\
    { \small\tt $^*$alessandra.menafoglio@polimi.it}\\
}

\date{}

\maketitle


\begin{abstract}
A new orthogonal decomposition for bivariate probability densities
embedded in Bayes Hilbert spaces is derived. It allows one to represent a
density into independent and interactive parts, the former being built
as the product of revised definitions of marginal densities and the
latter capturing the dependence between the two random variables being
studied. The developed framework opens new perspectives for dependence
modelling (which is commonly performed through copulas),
and allows for the analysis of dataset of bivariate densities, in
a Functional Data Analysis perspective. A spline representation
for bivariate densities is also proposed, providing a computational
cornerstone for the developed theory.
\end{abstract}




\section{Introduction}


The analysis of distributional data is gaining an increasing interest in the applied sciences. Distributional data, such as probability density functions (PDFs) or cumulative distribution functions, are routinely collected in social sciences (e.g., population pyramids \citep{delicado11,hron16} and geosciences (e.g., particle-size distributions \citep{menafoglio14,menafoglio16}). Analyses of distributional data based on methods designed for functional data in $L^2$ often lead to inappropriate results, such as negative predictions \citep{menafoglio18,Tal}.

It is now widely recognized that an appropriate statistical analysis of PDF data should be precisely based on their characterizing properties (e.g., \citep{Nerini07,Petersen16,hron16,menafoglio14}). In the literature, several approaches have been proposed to serve the purpose of analysing datasets of PDFs. Most works propose to analyse PDF data through a prior data transformation. For instance, \citep{delicado11} considers a transformation approach to the principal component analysis of a dataset of PDFs. \citep{srivastava07} use a square root transformations of densities to deal with a time-warping function in registration. \citep{Petersen16} propose a set of transformations that map the PDF data to a Hilbert space, where further statistical analyses are possible; this setting allows for, e.g., principal component analysis, classification, regression. A relatively large body of recent literature proposes the use of the Wasserstein metric to define a notion of distance for density data. Such metric has appealing interpretations, being related to the problem of optimal transport. However, it defines a non-linear space (Riemannian manifold), thus requiring the development of {\em ad hoc} methods for this setting (e.g., \citep{Bigot19,Panaretos19,Petersen19}), based on Frech\'{e}t statistics.
A different approach is that relying on the theory of Bayes linear spaces, that represent a generalization to the infinite-dimensional setting of the Compositional Data Analysis (CoDa, \citep{pawlowsky15}) approach. In this setting, PDFs are considered as infinite-dimensional objects that provide relative information \cite{boogaart10,boogaart14}. Bayes Hilbert spaces were built as to represent the so-called {\em principles} of CoDa (i.e., scale invariance, relative scale, sub-compositional coherence, see \citep{pawlowsky15}), through a Hilbert geometry for PDFs. The Hilbert structure of the space allows one to develop most methods of functional data analysis, while accounting for the peculiar nature of PDFs. These include principal component analysis \cite{hron16}, functional regression \cite{Tal}, spatial prediction \cite{menafoglio14}, profile monitoring \cite{menafoglio18}, time-series analysis \cite{kokoszka19,Seo19}.
Even though the statistical literature is nowadays well-developed for distributional data, little attention has been paid so far to the setting of multivariate densities, whose study is of paramount importance in the applications. A first contribution in direction of bivariate densities was recently provided in the preprint by \citep{iacopini19}, which uses the theory of Bayes Hilbert spaces over bivariate domains to study the temporal dynamic of coupled time series, modelled through copulas. As a key element of innovation with respect to previous literature, the present work proposes a novel statistical framework for bivariate PDFs, that allows studying the dependence between the target random variables, grounding on the geometry of the Bayes space. We provide new meaningful notions of compositional marginals (so-called \textit{geometric marginals}), which play the roles of the marginal distributions, consistent with the Bayes geometry. We further derive an orthogonal decomposition of bivariate PDFs in terms of independence and interaction parts, generalizing the well-known results developed in the discrete case (i.e., for compositional tables, \citep{egozcue08,egozcue15}). To allow for explicit computations of the marginals and of the latter representation, we develop a novel $B$-spline representation for bivariate PDFs, compatible with the compositional nature of the data.

The methodological results of our work shed light on the structure of multivariate Bayes spaces, suggesting a direction to create connections between the theory of Bayes spaces and the theory of copulas \cite{nelsen06}, which are widely used to build multivariate PDFs from marginals. Note that the theory of copulas is well-established and allows one to describe the joint distribution function of two random objects under very general assumptions. Our work is mostly focused on density functions (PDFs) instead, entailing a difference in the approaches in terms of (i) the assumptions made on the distribution at hand and (ii) the theoretical properties of the object being studied. In this sense, this work presents the initial steps in the direction of a new framework for dependence modeling, for which the extension to more general distributions (e.g., not absolutely continuous) is foreseen. On the other hand, this work is primarily aimed to build an analytical framework for datasets made of multivariate distributional objects, within the context of Functional Data Analysis (FDA, \citep{ramsay05}). In this view, building the dependence modelling on PDFs might be preferable, and this would be consistent with the usual practice in FDA, where regularity assumptions (continuity, boundedness, squared-integrability) are typically made on data. In this context, the appealing properties of the Bayes space approach which are discussed in this work (resulting, e.g., from the \emph{orthogonal} decomposition of the bivariate density into independent and interactive parts), are seen as key factors potentially fostering the development and interpretation of new FDA methods for multivariate distributional observations. In this sense, the methodology presented in the paper is going to offer an alternative viewpoint to the standard copula theory by providing an orthogonal decomposition of bivariate PDFs, while opening a novel frontier to analyse samples of bivariate densities using methods of FDA.

The remaining part of this work is organized as follows. In Section \ref{sec:Bayes}, the Bayes space methodology is recalled from \citep{boogaart14}, with particular reference to bivariate densities. This enables us to develop an orthogonal decomposition of bivariate densities into independent and interactive parts, thoroughly discussed in Section
\ref{sec:decomp} and demonstrated with simulated truncated Gaussian densities in Section \ref{sec:sim}. In Section \ref{sec:splines}, a spline representation for bivariate densities mapped in the $L^2$ space is introduced; such representation is relevant to allow processing raw data, and to develop efficient computational methods. In Section \ref{sec:anthropometric} the theoretical framework is applied to a time series of bivariate densities coming from an anthropometric cross-sectional study. The final Section \ref{sec:conc} concludes with some overview comments and further perspective.


\section{PDFs as elements of a Bayes space}\label{sec:Bayes}


Bayes spaces are designed to provide a geometrical representation for density functions characterized by the property of scale invariance \citep{boogaart14}. The latter property assumes that, given a domain $\Omega$ and a positive real multiple $c$, two proportional positive functions $f(x)$ and $g(x)$ (i.e., such that $g(x)=cf(x)$, for $c>0$) carry essentially the same, relative information \cite{boogaart14}. This follows also the common strategy used in Bayesian statistics where multiplying factors are typically dropped from computations, as these are not essential to the definition of the distributions at hand. Note that the scale invariance of a density $f$ is a direct consequence of the same property of the associated measure $\mu$, i.e., of the $\sigma$-finite measure $\mu$ such that $f=d\mu/d\mathsf{P}$ for a reference measure $\mathsf{P}$. In this context, we refer to the so-called $\mathcal{B}$-equivalence of measures (and densities): two measures $\mu$ and $\nu$ are $\mathcal{B}$-equivalent if they are proportional, i.e., there exists a positive real multiple $c$ such that $\nu({\sf A})=
c\cdot\mu({\sf A})$ for any $\sf{A}\in\mathcal{A}$, $\mathcal{A}$ being a sigma-algebra on $\Omega$.

Given a $\sigma$-finite measure $\mathsf{P}$, the Bayes space $\mathcal{B}^2(\mathsf{P})$ is a space of $\mathcal{B}$-equivalence classes of $\sigma$-finite positive measures $\mu$ with square-integrable log-density w.r.t. $\mathsf{P}$, i.e.,
$$
\mathcal{B}^2(\mathsf{P}) = \left\{ \mu \in \mathcal{B}^2(\mathsf{P}) : \int \left|\ln \frac{d \mu}{d \mathsf{P}} \right|^2  d \, \mathsf{P}< +\infty\right\}.
$$
From the practical point of view, an important role is played by the reference measure $\mathsf{P}$, as thoroughly investigated in \citep{Talska2019}.
The choice of the reference measure determines a weighting of the domain $\Omega$ of the PDF, which can be used to give more relevance to certain regions of $\Omega$ when conducting FDA, according to the purpose of the analysis \cite{boogaart14,egozcue16}. Given that the weighting of the domain is not of primary interest here and one would intuitively resort simply to the Lebesgue reference measure, the discussion on $\mathsf{P}$ might seem somehow lateral to the main focus of this work. Nevertheless, as we will see already in Theorem \ref{thm:ind-amar}, the scale of $\mathsf{P}$ indeed matters for a meaningful decomposition of a bivariate density into independent and interactive parts. For this reason, we here limit to mention two key points which shall be useful in the following. First, in general, an analysis based on a reference measure $\mathsf{P}$ does not provide the same results as an analysis based on $c\mathsf{P}$, for $c>0$. Indeed, using $\mathsf{P}$ or $c\mathsf{P}$ typically leads to a difference in the scale of the result.
Second, to change the reference measure from $\lambda$ to a measure $\mathsf{P}$ with strictly positive $\lambda$-density $p = d \mathsf{P}/ d\lambda$, the well-known chain rule can be used. For a generic measure $\mu$ one has
$$
\mu(\mbox{A})= \int_{\mbox{A}}\frac{d \mu}{d \lambda} \, d \lambda = \int_{\mbox{A}} \frac{d \mu}{d \lambda} \cdot \frac{d \lambda}{d \mathsf{P}}\, d \mathsf{P} = \int_{\mbox{A}} \frac{d \mu}{d \lambda} \cdot \frac{1}{p}\, d \mathsf{P}.$$

The Bayes space, as described above, can also be defined for the case when the domain $\Omega$ is a Cartesian product of two domains $\Omega_X$ and $\Omega_Y$, i.e., $\Omega=\Omega_X\times\Omega_Y$. In this case, the reference measure $\mathsf{P}$ can be decomposed as a product measure $\mathsf{P}=\mathsf{P}_X\times\mathsf{P}_Y$ and the Hilbert space structure of the Bayes space $\mathcal{B}^2(\mathsf{P})$ \cite{boogaart14,egozcue16} can be built accordingly. In this case, the operations of \textit{perturbation} and
\textit{powering} can be defined for any two bivariate densities $f,\,g$ with respect to $\mathsf{P}$, i.e., $f,\,g\in\mathcal{B}^2(\mathsf{P})$, and a real constant $\alpha$ as
\begin{equation*}
(f \oplus g)(x,y) =_{\mathcal{B}^2(\mathtt{P})} f(x,y) \cdot g(x,y) \quad \text{and}
\quad (\alpha \odot f)(x,y) =_{\mathcal{B}^2(\mathtt{P})} f(x,y)^{\alpha},
\end{equation*}
respectively. The lower index in $=_{\mathcal{B}^2(\mathtt{P})}$ means that the right hand side of the equations can be arbitrarily rescaled without altering the relative information that the resulting density in $\mathcal{B}^2(\mathsf{P})$ contains. The Hilbert space structure is completed by defining the inner product,
\begin{equation} \label{eq: inner}
\begin{split}
\left\langle f, g\right\rangle_{\mathcal{B}^2(\mathsf{P})} & = \frac{1}{2\mathtt{P}(\Omega)} \iint_\Omega \iint_\Omega \ln \frac{f(x,y)}{f(s,t)} \ln \frac{g(x,y)}{g(s,t)} \, d \mathtt{P}(x,y)\,d \mathtt{P}(s,t) =\\
& = \frac{1}{2\mathtt{P}(\Omega)} \int\limits_{\Omega_X} \int\limits_{\Omega_Y} \int\limits_{\Omega_X} \int\limits_{\Omega_Y} \ln \frac{f(x,y)}{f(s,t)} \ln \frac{g(x,y)}{g(s,t)} \,
d \mathtt{P}_X(x)\,d \mathtt{P}_Y(y)\,d \mathsf{P}_X(s)\,d \mathsf{P}_Y(t),
\end{split}
\end{equation}
which implies in the usual way also the norm and the distance,
\begin{equation} \label{eq:norm}
||f||_{\mathcal{B}^2(\mathsf{P})}=\sqrt{\langle f, f\rangle_{\mathcal{B}^2(\mathsf{P})}},\quad d_{\mathcal{B}^2(\mathsf{P})}(f,g)=||f\ominus g||_{\mathcal{B}^2(\mathsf{P})},
\end{equation}
where $f\ominus g=f\oplus[(-1)\odot g]$ is the perturbation-subtraction of densities.
Here, the definition of the inner product \eqref{eq: inner} is presented according to \citep{egozcue16}. While the scale of the reference measure $\mathsf{P}$ does not have any impact for the operations of perturbation and powering, it does influence the inner product because the scale corresponds to shrinkage (or expansion) of the Bayes space (for details, see \citep{Talska2019}).

The usual strategy when dealing with the Bayes spaces \cite{boogaart14,menafoglio14,hron16} is not to process densities directly in the original space but to map them into the standard $L^2$ space where most of the widely-used methods of functional data analysis (FDA, \citep{ramsay05}) can be employed. The \textit{clr transformation} of a bivariate density $f(x,y)\in\mathcal{B}^2(\mathsf{P})$ is a real function $f^c:\Omega\rightarrow \mathbb{R}$, $f^c\in L_0^2(\mathsf{P})$, defined -- using Fubini's theorem -- as
\begin{eqnarray}
\label{biclr}
f^c(x,y)&=&\ln f(x,y)-\frac{1}{\mathsf{P}(\Omega)}\iint_{\Omega}\ln f(x,y)\, d\mathsf{P}=\\
\nonumber &=& \ln f(x,y)-\frac{1}{\mathsf{P}(\Omega)}\int_{\Omega_X}\int_{\Omega_Y}\ln f(x,y)\,d\mathsf{P}_X d\mathsf{P}_Y.
\end{eqnarray}
Similarly as for perturbation and powering, the scale of $\mathsf{P}$ does not play any role in \eqref{biclr}, too. On the other hand, one should note that the resulting function $f^c$ is expressed with respect to reference $\mathsf{P}$. As a consequence, using any measure other than the Lebesgue $\lambda$ leads to clr-transformations defined over a weighted $L^2$ space $L^2(\mathsf{P})$ \cite{egozcue16}. Moreover, one should also take into account the zero-integral constraint of clr transformed densities, i.e.,
\begin{equation}
\label{zeroint}
\int_{\Omega_X}\int_{\Omega_Y}f^c(x,y)\,d\mathsf{P}_X d\mathsf{P}_Y=0.
\end{equation}
In the following, we shall indicate by $L_0^2(\mathsf{P})$ the subspace of the $L^2(\mathsf{P})$ space of (equivalence classes of) functions having zero integral; in particular, one clearly has that $f^c(x,y) \in L_0^2(\mathsf{P})$. Nevertheless, previous works focused on the univariate case demonstrate that this constraint usually does not represent any serious obstacle for the application of FDA methods, especially if a proper spline representation of the densities is used \cite{hron16,Mach,Tal}. Since a reliable and flexible spline representation forms a cornerstone in a large number of computational methods for FDA \cite{ramsay05}, we shall pay special attention in developing a bivariate $B$-splines basis suited to represent clr transformation of bivariate densities in Section \ref{sec:splines}.


\section{Decomposition of bivariate densities}\label{sec:decomp}


One of the key goals in probability theory is to study dependence structure between two random variables. A systematic approach to the analysis of dependence structure is represented by the theory of copulas \cite{nelsen06}, firstly introduced by Sklar \cite{sklar59}. The well-known Sklar's theorem provides a decomposition of any PDF into its interactive and independent parts, the latter being built as the product of the respective marginal PDFs. Relying on the Bayes space methodology allows one to provide a similar decomposition which is now, however, orthogonal. This important property enables for an elegant geometrical representation of the decomposition, and for a powerful probabilistic interpretation if a normalized reference measure is used, with direct consequences from the statistical viewpoint. For example, the proposed decomposition allows one to derive a measure of dependence called \textit{simplicial deviance}, defined as the squared norm of the density expressing (solely) relationships between both variables (factors).

The orthogonal decomposition of bivariate densities grounds on a novel definition of marginals, named \emph{geometric marginals}, which are built upon marginalizing the bivariate clr transformation as follows. Given $x\in\Omega_X$ and $y\in\Omega_Y$, we define the \emph{clr marginals} as
\begin{equation}
\label{xclrmar}
f^c_X(x)=\int\limits_{\Omega_Y}f^c(x,y)\,d\mathsf{P}_Y=\int\limits_{\Omega_Y}\ln f(x,y)\,d\mathsf{P}_Y-\frac{\mathsf{P}_X(\Omega_X)}{\mathsf{P}(\Omega)}\int\limits_{\Omega_X}\int\limits_{\Omega_Y}\ln f(x,y)\,d\mathsf{P}_X d\mathsf{P}_Y
\end{equation}
and
\begin{equation}
\label{yclrmar}
f^c_Y(y)=\int\limits_{\Omega_X}f^c(x,y)\,d\mathsf{P}_X=\int\limits_{\Omega_X}\ln f(x,y)\,d\mathsf{P}_X-\frac{\mathsf{P}_Y(\Omega_Y)}{\mathsf{P}(\Omega)}\int\limits_{\Omega_X}\int\limits_{\Omega_Y}\ln f(x,y)\,d\mathsf{P}_X d\mathsf{P}_Y,
\end{equation}
respectively. It is easily seen that $f_X^c\in L_0^2(\Omega_X)$ and $f_Y^c\in L_0^2(\Omega_Y)$, where $L_0^2(\Omega_i)$ stands for the subspace of $L^2(\Omega_i)$ whose elements integrate to zero. We define the \textit{geometric marginals} $f_X\in\mathcal{B}^2(\Omega_X)$ and $f_Y\in\mathcal{B}^2(\Omega_Y)$ as the elements of $B^2(\Omega_X)$ and $B^2(\Omega_Y)$ associated with the clr-marginals $f^c_X$, $f^c_Y$, respectively, i.e.,
\begin{eqnarray}
\label{gmarg}
f_X(x)&=_{\mathcal{B}(\mathsf{P}_X)}&\exp\left\{f^c_X(x)\right\}=_{\mathcal{B}(\mathsf{P}_X)} \exp\left\{\int_{\Omega_Y}\ln f(x,y)\,d\mathsf{P}_Y\right\},\\
\nonumber f_Y(y)&=_{\mathcal{B}(\mathsf{P}_Y)}&\exp\left\{f^c_Y(y)\right\}=_{\mathcal{B}(\mathsf{P}_Y)} \exp\left\{\int_{\Omega_X}\ln f(x,y)\,d\mathsf{P}_X\right\}.
\end{eqnarray}
In the following, the terms \emph{marginal}, \emph{$X$-marginal} and \emph{$Y$-marginal} will always refer to the geometric notion of marginals given in \eqref{gmarg}.

In probability theory, independence of random variables corresponds to the possibility of expressing a joint density as a product of its \emph{marginals}. In a setting where the latter are defined as the geometric marginals \eqref{gmarg}, the \textit{independent} and \textit{interactive parts} of $f(x,y)\in\mathcal{B}^2(\mathsf{P})$ can be defined, respectively, as
\begin{equation}
\label{indep}
f_{\mathrm{ind}}(x,y)=f_X(x)f_Y(y),\ (x,y)\in\Omega
\end{equation}
and
\begin{equation}
\label{inter}
f_{\mathrm{int}}(x,y)=\frac{f(x,y)}{f_X(x)f_Y(x)}=f(x,y)\ominus f_{\mathrm{ind}}(x,y),
\end{equation}
where $f_X(x)$ and $f_Y(y)$ are the geometrical marginals defined above.
The first and foremost important property of the proposed decomposition
\begin{equation}
\label{ddecomp}
f(x,y)=f_{\mathrm{ind}}(x,y)\oplus f_{\mathrm{int}}(x,y)
\end{equation}
for a bivariate density $f(x,y)$ is the \textit{orthogonality} its parts. In the following, the geometrical marginals will be formally taken as bivariate functions, i.e. $f_X(x)\equiv f_X(x,y)$ and $f_Y(y)\equiv f_Y(x,y)$, and considered as elements of $\mathcal{B}^2(\mathsf{P})$; similarly for their clr counterparts. This enables, among others, to express the independence density $f_{\mathrm{ind}}$ as sum (perturbation) of the geometric marginals, i.e.,
\begin{equation}
\label{inddecomp}
f_{\mathrm{ind}}(x,y)=f_X(x,y)\oplus f_Y(x,y).
\end{equation}

\begin{theorem}\label{thm:ortho}
For the independent and interactive parts of a bivariate density $f(x,y)$, it holds that\\
{\bf (i)} $\langle f_{\mathrm{ind}},f_{\mathrm{int}}\rangle_{\mathcal{B}^2(\mathsf{P})}=0$, or, equivalently that\\
{\bf (ii)}  $\langle f^c_{\mathrm{ind}},f^c_{\mathrm{int}}\rangle_{L_0^2(\mathsf{P})}=0$.
\end{theorem}

The proof of Theorem \ref{thm:ortho} -- as well as those of the following theorems -- is reported in Supplementary Material. Note that, from the orthogonality of the decomposition $f=f_{\mathrm{ind}}\oplus f_{\mathrm{int}}$, the Pythagorean theorem follows directly, i.e.,
$||f||_{\mathcal{B}^2(\mathsf{P})}^2=||f_{\mathrm{ind}}||_{\mathcal{B}^2(\mathsf{P})}^2+||f_{\mathrm{int}}||_{\mathcal{B}^2(\mathsf{P})}^2$.

A further important property of independence densities $f_{\mathrm{ind}}$ is the following. Call \textit{arithmetic marginals} the usual marginal distributions (a similar notation being used in the discrete case of compositional tables \cite{egozcue15})
$$f_{X,a}(x) = \int_{\Omega_Y} f (x,y) d\mathsf{P}_Y, \quad f_{Y,a}(y) = \int_{\Omega_X} f(x,y) d\mathsf{P}_X. $$
It is clear that, if the theory were built on arithmetic marginals, the above decompositions (\ref{ddecomp}) and (\ref{inddecomp}) together with the statement of Theorem \ref{thm:ortho} would not be achieved. On the other hand, there is an interesting link between the two types of marginals (geometric or arithmetic) when the independent part is concerned. Indeed, the following result states that, whenever the random variables $X,Y$ are independent, the bivariate PDF coincides with its independent part defined in \eqref{indep}. In this case, if the reference measure is a probability measure (i.e., it is normalized), the arithmetic and the geometric marginals coincide.

\begin{theorem}\label{thm:ind-amar}
Let $f$ be an independence density and let the reference measure $\mathsf{P}=\mathsf{P}_X\times\mathsf{P}_Y$ be the product measure of probability measures $\mathsf{P}_X,\,\mathsf{P}_Y$. Then the arithmetic and geometric marginals of $f$ coincide.
\end{theorem}
As such, the independent part built through the geometric marginals enables one to fully capture the joint distribution of two random variables when these are independent.

The next theorem states the mutual orthogonality between the geometric marginals ($f_X(x,y)$ and $f_Y(x,y)$) and the \textit{interaction density}
$f_{\mathrm{int}}(x,y)$.

\begin{theorem}\label{thm:ortho-mar}
The $X$-marginal and $Y$-marginal are orthogonal with respect to the Bayes space $\mathcal{B}^2(\mathsf{P})$, i.e., $\langle f_X, f_Y \rangle_{\mathcal{B}^2(\mathsf{P})} = 0$. Moreover, the marginals are also orthogonal to the interaction density, i.e., $\langle f_X, f_{\mathrm{int}} \rangle_{\mathcal{B}^2(\mathsf{P})} = 0$ and $\langle f_Y, f_{\mathrm{int}} \rangle_{\mathcal{B}^2(\mathsf{P})} = 0$.
\end{theorem}

The relations $\langle f^c,f_X^c\rangle_{L_0^2(\mathsf{P})}=||f_X^c||^2_{L_0^2(\mathsf{P})}$,
$\langle f^c,f_Y^c\rangle_{L_0^2(\mathsf{P})}=||f_Y^c||^2_{L_0^2(\mathsf{P})}$ and $\langle f_X^c,f_Y^c\rangle_{L_0^2(\mathsf{P})}=0$ nicely illustrate that
$X$- and $Y$- marginals of the density $f(x,y)$ represent its orthogonal projections. In addition, the Pythagorean theorem between the independence density and its projections holds, $||f_{\mathrm{ind}}||_{\mathcal{B}^2(\mathsf{P})}^2=||f_X||_{\mathcal{B}^2(\mathsf{P})}^2+||f_Y||_{\mathcal{B}^2(\mathsf{P})}^2$.

As a consequence of Theorems \ref{thm:ind-amar}-\ref{thm:ortho-mar}, one can conclude that, in case of independence, arithmetic and geometric marginals coincide, and the interaction part is null (i.e., it is the neutral element of perturbations).  More in general, the next result states that the geometric marginals are completely determined by the independent part of the bivariate density. Here, the clr marginals of $f_{\mathrm{int}}$ are defined as
\begin{eqnarray*}
f^c_{\mathrm{int},X} &=& \int_{\Omega_Y} \ln f_{\mathrm{int}}(x,y) d\mathsf{P}_Y - \frac{\mathsf{P}_Y(\Omega_Y)}{\mathsf{P}(\Omega)} \iint_{\Omega} f_{\mathrm{int}}(x,y) d\mathsf{P}, \\
f^c_{\mathrm{int},Y} &=& \int_{\Omega_X} \ln f_{\mathrm{int}}(x,y) d\mathsf{P}_X - \frac{\mathsf{P}_X(\Omega_X)}{\mathsf{P}(\Omega)} \iint_{\Omega} f_{\mathrm{int}}(x,y) d\mathsf{P},
\end{eqnarray*}
and the geometric marginals $f_{\mathrm{int},X}$,$f_{\mathrm{int},Y}$ are the associated densities in $\mathcal{B}^2(\mathsf{P})$.

\begin{theorem}\label{thm:int-neutral}
Whenever the reference measure is the product measure of probability measures $\mathsf{P}_X$, $\mathsf{P}_Y$, the geometric marginals $f_{\mathrm{int},X}$,$f_{\mathrm{int},Y}$ of the interaction part $f_{\mathrm{int}}$ coincide with the neutral element of perturbation, i.e., for any $f$ in $\mathcal{B}^2(\mathsf{P})$ one has
\begin{equation}\label{eq:thm3}
f \oplus f_{\mathrm{int},X} = f;\quad f\oplus  f_{\mathrm{int},Y} = f.\end{equation}
\end{theorem}

Theorem \ref{thm:int-neutral} motivates the name \textit{interaction} density. Indeed, decomposition \eqref{ddecomp} applied to $f_{\mathrm{int}}$ reads
\[
f_{\mathrm{int}}=0_{\oplus}\oplus f_{\mathrm{int}},
\]
where $0_{\oplus}$ is the neutral element of perturbation (with respect to probability reference measure $\mathsf{P}$). Accordingly, the independent part of an interaction density is the null element $0_{\oplus}$. On the other hand, for an independent density, the interaction part is null. More in general, for any bivariate density $f$, the \emph{nearest} independence density is $f_{\mathrm{ind}}$, and its distance from it is precisely $||f_{\mathrm{int}}||_{\mathcal{B}^2(\mathsf{P})}$.
The squared norm $||f_{\mathrm{int}}||^2_{\mathcal{B}^2(\mathsf{P})}$ can be thus taken as a proper measure of dependence. For consistency with the discrete case \cite{egozcue15}, we shall name it \textit{simplicial deviance}, $\Delta^2(f)=||f_{\mathrm{int}}||^2_{\mathcal{B}^2(\mathsf{P})}$. Dividing the simplicial deviance by the squared norm of the bivariate density, one obtains a relative measure of dependence, hereafter named \textit{relative simplicial deviance},
\begin{equation}
R^2(f)=\frac{||f_{\mathrm{int}}||^2_{\mathcal{B}^2(\mathsf{P})}}{||f||^2_{\mathcal{B}^2(\mathsf{P})}},\quad 0\leq R^2(f)\leq 1.
\end{equation}

Note that $R^2(f)$ captures the amount of information contained in the interaction part with respect to the overall information within the density. If $R^2(f)$ is small ($R^2(f)\sim 0$), it means that most of the density is described by the independent part, and vice versa. A further advantage of the use of $R^2(f)$ is its relative character: $R^2(f)$ does not rely on the norm of the bivariate density which might be in practice influenced by the sample size of data being aggregated in the density.

Further, it can be proven that $f_{\mathrm{int}}$ is \textit{marginal invariant}, i.e., when the bivariate density $f$ is perturbed \emph{marginally} (i.e., by marginal densities $g_X$ and $g_Y$), the interaction part $f_{\mathrm{int}}$ is not changed. This important property \cite{yule12} is formulated in the next theorem.

\begin{theorem}\label{thm:marg-pert}
Let $\mathsf{P}=\mathsf{P}_X\times\mathsf{P}_Y$ be a probability measure, $f\in\mathcal{B}^2(\mathsf{P})$ a bivariate density with the orthogonal decomposition
$f=f_{\mathrm{ind}}\oplus f_{\mathrm{int}}$ and $g_X$, $g_Y$ marginal densities, in the sense that these latter are bivariate densities in $\mathcal{B}^2(\mathsf{P})$,  constant in one argument, i.e.,
\[
g_X(x,y) = \widetilde{g}_X(x),\quad g_Y(x,y) = \widetilde{g}_Y(y),\quad (x,y)\in\Omega. \quad\]
Then, the marginally perturbed density, $h=g_X\oplus g_Y\oplus f$, has the orthogonal decomposition $h=h_{\mathrm{ind}}\oplus h_{\mathrm{int}}$, where $h_{\mathrm{int}}=f_{\mathrm{int}}$ and $h_{\mathrm{ind}}=g_X\oplus g_Y\oplus f_{\mathrm{ind}}$.
\end{theorem}


\section{An example with a truncated Gaussian Density}\label{sec:sim}


For the sake of illustration, we present an example of application of the proposed framework to densities in the Gaussian family, where computations can be made explicitly. Given that, in general, one may not expect to be able to perform this type of computations explicitly, in Section \ref{sec:splines} we develop a B-spline basis representation for bivariate distributions, from which the interactive and independent parts can be directly computed. We first consider a univariate Gaussian density, similarly as in \cite{hron16,delicado11}. For the sake of simplicity, we set the reference measure to the Lebesgue measure, and consider a zero-mean Gaussian density, truncated over the interval $I=[-T,T],\, T=5$. In this case, the density $f$ reads
\[
f(x) =_{\mathcal{B}^2(\lambda)} \exp\left\{-\frac{x^2}{2\sigma^2}\right\}, \quad x\in I.\]
The (univariate) clr-transformation of $f$ is defined as
\[
f^c(x) = -\frac{x^2}{2\sigma^2} + \frac{T^2}{6\sigma^2}, \quad x\in I.\]

Increasing the dimensionality of the sample space, we consider a zero-mean bivariate Gaussian density $\mathcal{N}_2(\mu, \Sigma)$ with respect to the (product) Lebesgue measure $\lambda[I] = \lambda[I_1]\times \lambda[I_2]$, truncated on a rectangular domain $I = I_1 \times I_2 \subset \mathbb{R}^2$, with $I_1=I_2=[-T,T],\, T=5$. In this case, the density is defined, for $\mathbf x=(x,y) \in I$, as
\[
f(x,y) =_{\mathcal{B}^2(\mathsf{P})} \exp\left\{\mathbf {x}^{\top}\Sigma^{-1}\mathbf{x}\right\} = \exp\left\{-\frac{1}{2(1-\rho^2)}\left[\frac{x^2}{\sigma_1^2} - 2\rho\frac{xy}{\sigma_1\sigma_2} + \frac{y^2}{\sigma_2^2}\right]\right\}, \]
with $\sigma_i^2 = \Sigma_{ii}$ and $\rho\in[0,1]$ being the correlation coefficient. In this setting, the clr transformation of $f$ is
\[
f^c(x,y) = -\frac{1}{2(1-\rho^2)}\left[\frac{x^2}{\sigma_1^2} - 2\rho\frac{xy}{\sigma_1\sigma_2} + \frac{y^2}{\sigma_2^2}\right] + \frac{T^2}{6(1-\rho^2)}\left(\frac{1}{\sigma_1^2}+\frac{1}{\sigma_2^2}\right).\]
Marginalizing the clr transformation with respect to $x$ and $y$ yields the clr-marginals
\begin{eqnarray*}
f_X^c(x) &=& -\frac{1}{2(1-\rho^2)}\cdot\frac{2Tx^2}{\sigma_1^2}+\frac{T^3}{3(1-\rho^2)}\cdot\frac{1}{\sigma_1^2}, \quad x\in I_1,\\
f_Y^c(y) &=& -\frac{1}{2(1-\rho^2)}\cdot\frac{2Ty^2}{\sigma_2^2}+\frac{T^3}{3(1-\rho^2)}\cdot\frac{1}{\sigma_2^2}, \quad y\in I_2.
\end{eqnarray*}
On this basis, the geometric marginals are easily obtained -- following \eqref{xclrmar} and \eqref{yclrmar} -- as
\begin{eqnarray*}
f_X(x) &=_{\mathcal{B}^2(\mathsf{P})}& \exp\left\{-\frac{1}{2(1-\rho^2)}\cdot\frac{2Tx^2}{\sigma_1^2}\right\}, \quad x\in I_1,\\
f_Y(y) &=_{\mathcal{B}^2(\mathsf{P})}& \exp\left\{-\frac{1}{2(1-\rho^2)}\cdot\frac{2Ty^2}{\sigma_2^2}\right\}, \quad y\in I_2.
\end{eqnarray*}
Note that both marginals still belongs to a Gaussian family, with parameters $\mu_X=\mu_Y = 0$ and $\sigma_X^2 = \sigma_1^2(1-\rho^2) / 2T$,
$\sigma_Y^2 = \sigma_2^2(1-\rho^2) / 2T$.

Given the marginals, the independence and interactive parts are built as in \eqref{indep} and \eqref{inter}, leading to
\begin{eqnarray*}
f_{\mathrm{ind}} (x,y) &=_{\mathcal{B}^2(\mathsf{P})}& \exp\left\{-\frac{2T}{2(1-\rho^2)}\left[\frac{x^2}{\sigma_1^2} + \frac{y^2}{\sigma_2^2}\right]\right\};\\
f_{\mathrm{int}} (x,y) &=_{\mathcal{B}^2(\mathsf{P})}& \exp\left\{-\frac{1}{2(1-\rho^2)}\left[\frac{(1-2T)x^2}{\sigma_1^2} - 2\rho\frac{xy}{\sigma_1\sigma_2} + \frac{(1-2T)y^2}{\sigma_2^2}\right]\right\}.
\end{eqnarray*}
The clr transformations of the latter parts are found as
\begin{equation*}
\begin{split}
f^c_{\mathrm{ind}}(x,y)  =  & -\frac{T}{1-\rho^2}\left(\frac{x^2}{\sigma_1^2} + \frac{y^2}{\sigma_2^2}\right) + \frac{T^3}{3(1-\rho^2)}\left(\frac{1}{\sigma_1^2} + \frac{1}{\sigma_2^2}\right); \\
f^c_{\mathrm{int}} (x,y) = &  -\frac{1}{2(1-\rho^2)}\left(\frac{(1-2T)x^2}{\sigma_1^2} - 2\rho\frac{xy}{\sigma_1\sigma_2} + \frac{(1-2T)y^2}{\sigma_2^2}\right) +\\
& + \frac{T^2(1-2T)}{6(1-\rho^2)}\left(\frac{1}{\sigma_1^2} + \frac{1}{\sigma_2^2}\right).
\end{split}
\end{equation*}
Note that, in case of independence ($\rho=0$),
\[
f^c_{\mathrm{int}}(x,y) = \frac{3}{2}\left(\frac{3x^2+T^2}{\sigma_1^2} + \frac{3y^2+T^2}{\sigma_2^2}\right),\]
which is non-zero. This does not stand in contradiction with Theorem \ref{thm:ind-amar}, since the previous computations are indeed referred to the Lebesgue measure, which is not a probability measure (it is not normalized). Analogous computations made in the case of a uniform measure (i.e., the product measure $U[I]$ built upon uniform measures $\mathsf{P}_1 = U[I_1]$, $\mathsf{P}_2 = U[I_2]$) lead to a null $f^c_{\mathrm{int}}$ for $\rho=0$. Indeed, in this case one has that the clr geometric marginals are defined as
\begin{eqnarray}
\label{marg-Gaus-x}f_X^c(x) &=& -\frac{1}{2(1-\rho^2)}\cdot\frac{x^2}{\sigma_1^2}+\frac{T^2}{6(1-\rho^2)}\cdot\frac{1}{\sigma_1^2}, \quad x\in I_1,\\
\label{marg-Gaus-y}f_Y^c(y) &=& -\frac{1}{2(1-\rho^2)}\cdot\frac{y^2}{\sigma_2^2}+\frac{T^2}{6(1-\rho^2)}\cdot\frac{1}{\sigma_2^2}, \quad y\in I_2,
\end{eqnarray}
leading to the following forms for the independent and interaction clr-densities
\begin{eqnarray*}
f^c_{\mathrm{ind}}(x,y) &=& -\frac{1}{2(1-\rho^2)}\left(\frac{x^2}{\sigma_1^2} + \frac{y^2}{\sigma_2^2}\right) + \frac{T^2}{6(1-\rho^2)}\left(\frac{1}{\sigma_1^2} + \frac{1}{\sigma_2^2}\right);\\
f^c_{\mathrm{int}} (x,y) &=& \frac{1}{(1-\rho^2)}\left(\rho\frac{xy}{\sigma_1\sigma_2}\right).
\end{eqnarray*}
It is then clear that the interaction part precisely captures the terms in $f^c$ depending on the mixed polynomial $xy$ (i.e., the interaction between $x$ and $y$), and its magnitude is controlled by the magnitude of $\rho$. In case of independence ($\rho=0$), $f^c_{\mathrm{int}}$ is null, and $f_{\mathrm{int}}=0_\oplus$. Moreover, in this case, the geometric marginals and the arithmetic marginals coincide. Note that the former are found by normalizing the exponential of the first terms of $f_X^c$ and $f_Y^c$ in \eqref{marg-Gaus-x}-\eqref{marg-Gaus-y}. In the degenerate case of a perfect linear dependence between the marginal variables $X$ and $Y$ ($|\rho|=1$), $f^c_{\mathrm{int}}$ is indeed degenerate as well. In fact, for $|\rho| = 1$ not only $f^c_{\mathrm{int}}$ is not defined, but $f$ does not belong to $\mathcal{B}^2(U[I])$, nor to $\mathcal{B}^2(\lambda[I])$ (the logarithms of the corresponding densities are not in $L^2(U[I])$ nor in $L^2(\lambda[I])$).

Figure \ref{fig:Gaus-leb} reports the contour plots associated with the bivariate Gaussian density with $\sigma_1 = 2$, $\sigma_2 = 3$ and $\rho = 0.75$, when the reference measure is the Lebesgue measure. Figure \ref{fig:Gaus-uni} reports the analogue contour plots when the quantities are computed w.r.t. a Uniform measure. For the sake of clarity, quantities referred to the Uniform reference are reported with a subscript $w$ in Figure \ref{fig:Gaus-uni}. The figures clearly show that the scale of the reference measure plays indeed a role, particularly for the shape of $f_{\mathrm{int}}$ (Figures \ref{fig:Gaus-leb}c-f and \ref{fig:Gaus-uni}c-f). This is in agreement with the conclusions of \citep{Talska2019}, where the effect of the reference measure on the geometry of (univariate) Bayes spaces is discussed. Given the statistical consequences of Theorems \ref{thm:ind-amar} and \ref{thm:int-neutral}, the representation based on a normalized reference shall be here preferred. In the latter case (Figure \ref{fig:Gaus-uni}), the independent part represents the (unique) distribution which would be built upon the geometric marginals -- $f_X$ being $\mathcal{B}^2$-equivalent to a truncated $N(0,4(1-0.75^2))$, and $f_Y$ the $\mathcal{B}^2$-equivalent to a truncated $N(0,9(1-0.75^2))$. The simplicial deviance $\Delta^2(f)=\|f_{\mathrm{int}}\|^2_{\mathcal{B}^2(\mathsf{P})}$ is in this case $\Delta^2(f) = 5.66$. The value of the relative simplicial deviance $R^2(f)$ represents the proportion of the norm of $f$ which can be attributed to the interaction part (i.e., to the deviation from independence). In this example, such proportion is 51\%, indicating that the dependence between the two variables is indeed relevant in the definition of the bivariate distribution.

\begin{figure}[ht]
  \centering
  \includegraphics[width=.95\textwidth]{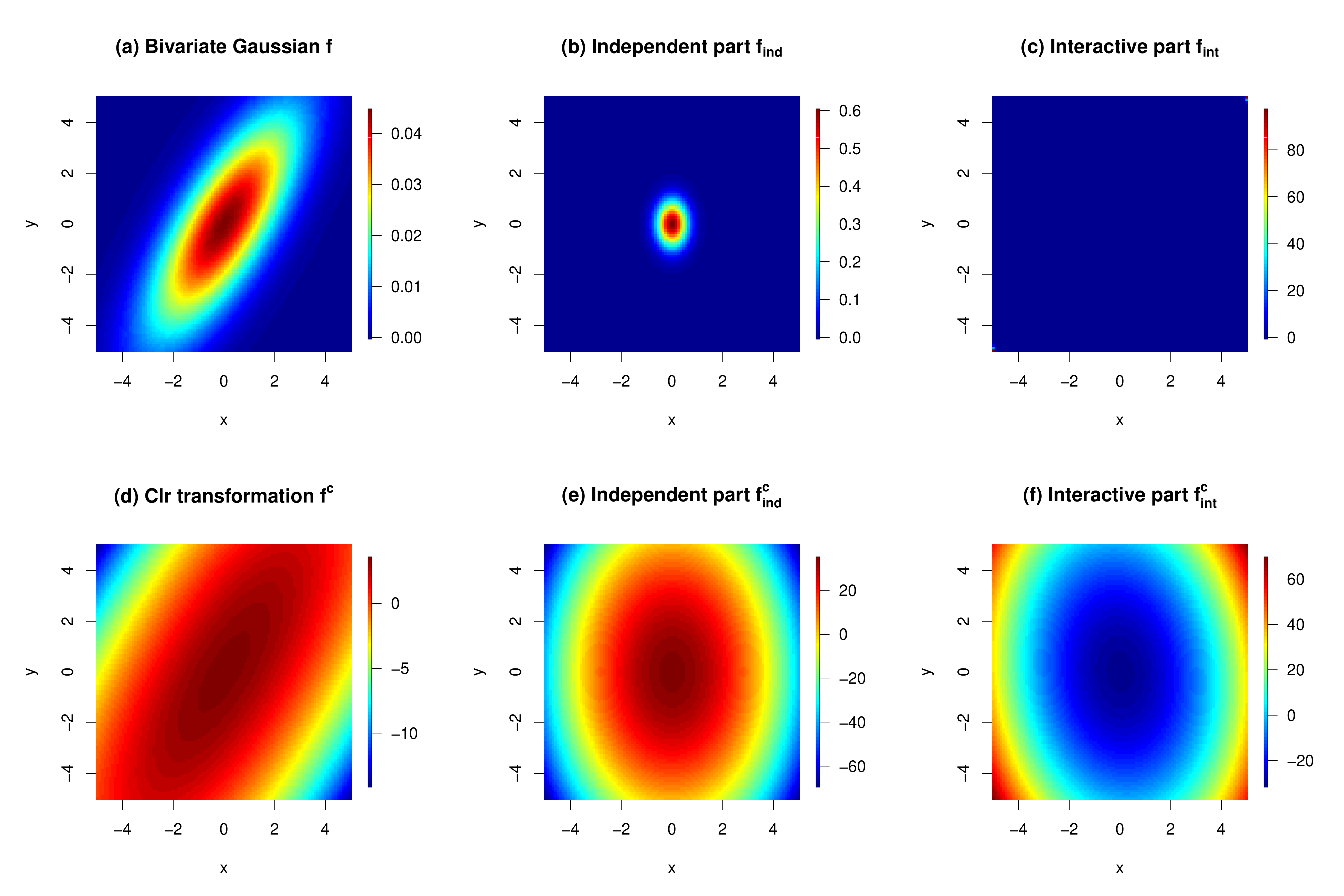}
  \caption{A simulated example with a truncated Gaussian density, with respect to the Lebesgue product measure $\lambda[I] = \lambda[I_1]\times \lambda[I_2]$, with $I_1 = I_2 = [-5,5]$, $\sigma_1 = 2$, $\sigma_2=3$, $\rho=0.75$.}\label{fig:Gaus-leb}
\end{figure}

\begin{figure}[ht]
  \centering
  \includegraphics[width=.95\textwidth]{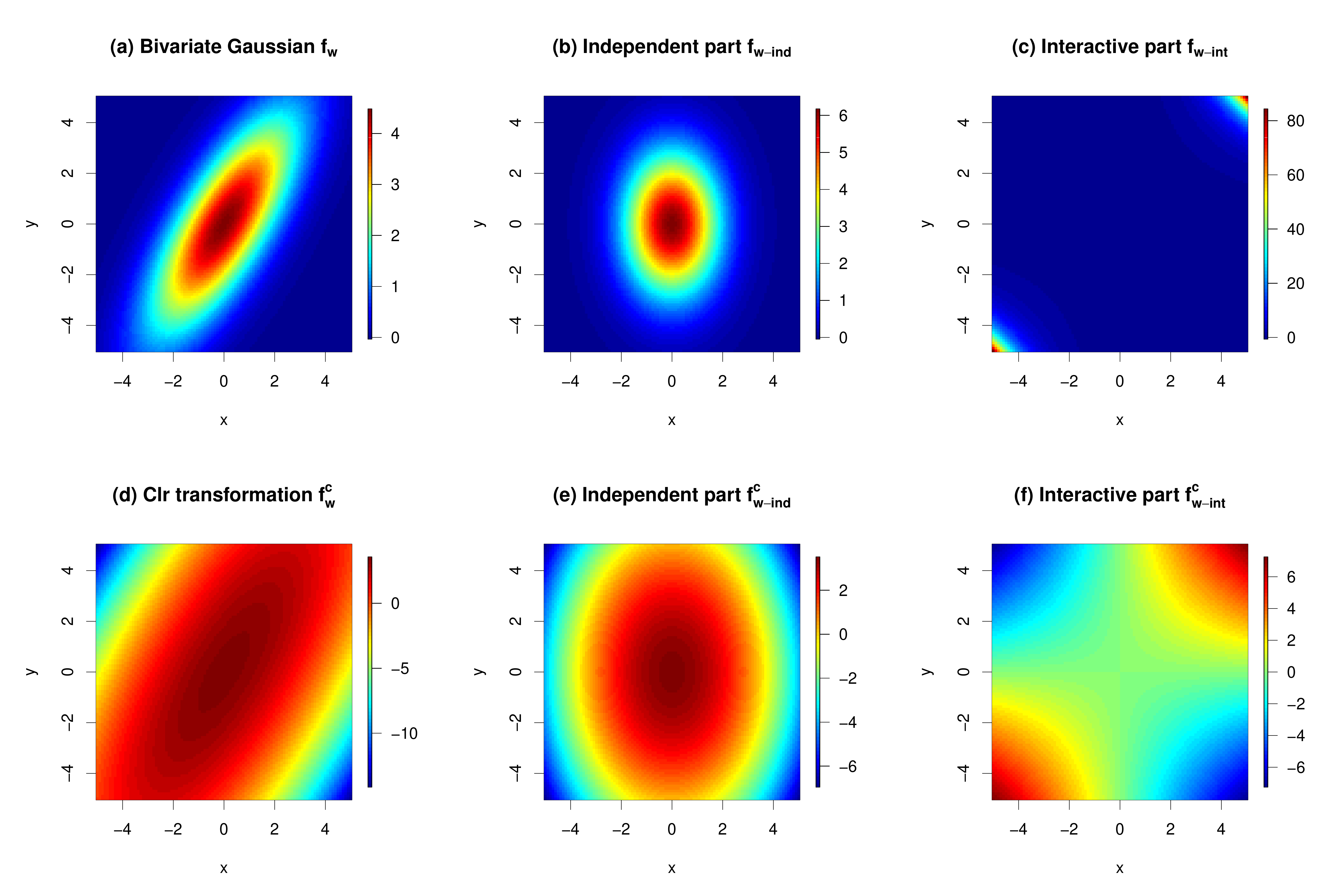}
  \caption{A simulated example with a Gaussian density, with respect to the Uniform product measure $\mathsf{P} = U[I_1]\times U[I_2]$, with $I_1 = I_2 = [-5,5]$, $\sigma_1 = 2$, $\sigma_2=3$, $\rho=0.75$.}\label{fig:Gaus-uni}
\end{figure}


\section{A spline representation for bivariate densities and their decompositions}\label{sec:splines}


Computational methods of FDA for the statistical analysis of datasets of bivariate densities are often based on basis representations for the data. In this section, we develop a spline representation for densities, which is based on a $B$-spline approximation for clr transformed data. This will allow for the smoothing of bivariate splines, and the direct computations of geometric marginals, independence and interaction parts, as well as of the relative simplicial deviance. On one hand, this avoids the necessity of developing splines directly in $\mathcal{B}^2(\mathsf{P})$; on the other one, it implies that the zero integral constraint needs to be taken into account.

This goal is here achieved by using tensor product splines \citep{deboor78,dierckx93,Schum} which are an established tool in the field and, in principle, enable a generalization to $k$ dimensions. However, for the purpose of this paper and for ease of notation, we shall focus on bivariate splines only. We also avoid considering the general reference measure $\mathsf{P}$ and focus on the Lebesgue measure (or its normalized counterpart, the uniform measure), although the case of a generic $\mathsf{P}$-reference can be reformulated as well \cite{Talska2019}.
We shall base our developments on \citep{Mach1,Mach2} -- the same representation being used for approximation of PDFs, e.g., in \citep{Mach,hron16,Tal,MenafoglioEtAl2018}; this setting is recalled in the Supplementary Material. Hereafter in this section, we limit to present the key points and results of our construction, leaving the details and proofs to the Supplementary Material.

\smallskip
We consider two strictly increasing sequences of knots
\begin{align}
  \Delta\lambda\, := \, \{\lambda_i\}_{i=0}^{g+1}, \quad  \lambda_{0}=a<\lambda_{1}<\ldots<\lambda_{g}<b=\lambda_{g+1} \label{knots_x},\\
  \Delta\mu \, := \, \{\mu_j\}_{j=0}^{h+1}, \quad \mu_{0}=c<\mu_{1}<\ldots<\mu_{h}<d=\mu_{h+1} \label{knots_y}
\end{align}
and denote by ${\cal S}_{kl}^{\Delta\lambda,\Delta\mu}(\Omega)$ the vector space of tensor product splines on $\Omega=[a,b]\times [c,d]$ of degree $k>0$ in $x$ and $l>0$ in $y$, with knots $\Delta\lambda$ in the $x$-direction and $\Delta\mu$ in $y$-direction. As usual in spline theory, to get a unique representation additional knots are considered, namely
\begin{alignat}{2}
\lambda_{-k} = \cdots = \lambda_{-1} = \lambda_{0} & = a, & \qquad b & = \lambda_{g+1} = \lambda_{g+2} = \cdots = \lambda_{g+k+1},\label{add_knots_x1}\\
\mu_{-l} = \cdots = \mu_{-1} = \mu_{0} & = c,  & \qquad d & = \mu_{h+1} = \mu_{h+2} = \cdots = \mu_{h+l+1}.\label{add_knots_y1}
\end{alignat}
The general goal here explored is that of smoothing the values $f_{i j}$ at points $(x_{i},y_{j})\in\Omega=[a, b]\times[c,d]$, $i=1,\ldots,n$, $j=1,\ldots,m$
using a tensor-product spline. The values $f_{i j}$ will be the clr-transformation of a discrete representation of the bivariate densities (i.e., histogram data), as showcased in Section \ref{sec:anthropometric}.
For the strictly increasing sequences of knots (\ref{knots_x}) and (\ref{knots_y}), a parameter $\alpha\in(0,1)$ and arbitrary
$u\in\left\{0,1,\ldots,k-1\right\}$ and $v\in\left\{0,1,\ldots,l-1\right\}$, we aim to find a spline  $s_{k l}(x, y)\in {\cal S}_{k l}^{\Delta\lambda,\Delta\mu}(\Omega)$
which minimizes the functional
\begin{equation}\label{smoothR2a}
    J_{uv}(s_{k l}) \, = \, \alpha\sum\limits_{i=1}^{n} \sum\limits_{j=1}^m \, \left[f_{i j}-s_{k l}(x_{i},y_j)\right]^{2} + (1-\alpha)\iint\limits_{\Omega}\left[s_{k l}^{(u,v)}(x,y)\right]^{2}\, \mbox{d}x\, \mbox{d}y,
\end{equation}
where the upper index $(u, v)$ stands for the derivative, specifically
$$
s_{k l}^{(u,v)}(x,y) \; = \; \frac{\partial^{u}}{\partial x^{u}} \, \frac{\partial^{v}}{\partial y^{v}} \; s_{k l}(x, y).
$$
Clearly, the choice of the parameter $\alpha$ and of the derivative orders $(u,v)$ affects the smoothness of the resulting spline. For the optimal choice of $\alpha$, the generalized cross-validation (GCV) criterion is used here, similarly as in \citep{machalova19}.

Given that we aim to reconstruct clr-transformed PDFs, the zero-integral constraint needs to be incorporated into the tensor product splines. Accordingly, we here aim to find a spline
$s_{k l}(x, y)\in {\cal S}_{k l}^{\Delta\lambda,\Delta\mu}(\Omega)$, $\Omega \, = \, [a,b] \times [c,d]$,
which minimizes the functional (\ref{smoothR2a}) and satisfies the additional condition
\begin{equation}\label{cond2}
\iint \limits_{\Omega} s_{kl}(x,y) \, \mbox{d}x \, \mbox{d}y \, = \, 0.
\end{equation}
We thus generalize to tensor product splines the idea presented in \citep{Mach} for one-dimensional splines.
To state the solution of the problem, and the conditions for its well-posedness, we need to introduce additional notation, that follows.

We express the tensor spline $s_{k l}(x,y)$ appearing in \eqref{smoothR2a} as
\begin{equation}\label{spline-biv}
  s_{kl}(x,y) \; = \; \sum\limits_{i=-k}^{g} \sum\limits_{j=-l}^{h} \, b_{ij} \, B_i^{k+1}(x) \, B_j^{l+1}(y),
\end{equation}
where $B_i^{k+1}(x)$, $B_j^{l+1}(y)$ are (univariate) $B$-splines defined on the sequence of knots $\{\lambda_i\}$ or $\{\mu_j\}$ and $b_{ij}$
are the coefficients of this spline. The tensor spline in
(\ref{spline-biv}) can be expressed in matrix notation as $
s_{kl}(x,y) \; = \; \mathbf{B}_{k+1}(x)\, \mathbf{B} \, \mathbf{B}_{l+1}^{\top}(y)$, where
$\mathbf{B}$ is a matrix of $B$-spline coefficients
$\mathbf{B}=\left(b_{ij}\right)_{i=-k,j=-l}^{g,h}$,
$\mathbf{B}_{k+1}(x)= \left(B_{-k}^{k+1}(x),\cdots,B_g^{k+1}(x)\right)$ is the collocation matrix of the $B$-splines $B_i^{k+1}(x)$, and
$\mathbf{B}_{l+1}(y)= \left(B_{-l}^{l+1}(y),\cdots,B_h^{l+1}(y)\right)$ is the collocation matrix of the $B$-splines $B_j^{k+1}(y)$. This admits also a tensor product representation, as
\begin{equation}\label{spline_tensor}
\begin{split}
s_{kl}(x,y) &  = \mathbf{B}_{k+1}(x) \, \mathbf{B} \, \mathbf{B}_{l+1}^{\top}(y)=
\left(\mathbf{B}_{l+1}(y) \otimes \mathbf{B}_{k+1}(x)\right) cs(\mathbf{B} ) = \\
& = { \mathbb B}(x,y) \, \, cs(\mathbf{B}),
\end{split}
\end{equation}
where
$
\mathbb{B}(x,y) \, := \mathbf{B}_{l+1}(y) \otimes \mathbf{B}_{k+1}(x)
$
and $cs(\mathbf{B})$ is the vectorized form of the matrix $\mathbf{B}$ (columnwise).

Let $\mathbf{S}_{u} = \mathbf{D}_{u}^x \mathbf{L}_{u}^x \cdots \mathbf{D}_{1}^x \mathbf{L}_{1}^x \in \mathbb{R}^{g+k+1-u,g+k+1}$, with $\mathbf{D}_j^x \; = \; (k+1-j)\, \mbox{diag} (d_{-k+j}^x,\ldots,d_g^x)$, $d_i^x \; = \; \dfrac{1}{\lambda_{i+k+1-j}-\lambda_i}$, $i=-k+j,\ldots,g$, and
\begin{equation*}
\mathbf{L}_j^x \; = \; \left( \begin{array}{cccc}
 -1 & 1      &.       & \\
    & \ddots & \ddots & \\
    &        & -1     & 1
        \end{array}\right)\in\mathbb{R}^{g+k+1-j,g+k+2-j}.
\end{equation*}
Similarly, let $\mathbf{S}_{v} = \mathbf{D}_{v}^y \mathbf{L}_{v}^y \cdots \mathbf{D}_{1}^y \mathbf{L}_{1}^y \in \mathbb{R}^{h+l+1-v,h+l+1}$, with $\mathbf{D}_j^y \; = \; (l+1-j)\, \mbox{diag} (d_{-l+j}^y,\ldots,d_h^y)$, $d_i^y \; = \; \dfrac{1}{\mu_{i+l+1-j}-\mu_i}$, $i=-l+j,\ldots,h$, and
\begin{equation*}
\mathbf{L}_j^y \; = \; \left( \begin{array}{cccc}
 -1 & 1      &.       & \\
    & \ddots & \ddots & \\
    &        & -1     & 1
        \end{array}\right)\in\mathbb{R}^{h+l+1-j,h+l+2-j}.
\end{equation*}

Denote by $\mathbb{S}$ the tensor product between $\mathbf{S}_{v}$ and $\mathbf{S}_{u}$, $\mathbb{S} := \mathbf{S}_{v} \otimes \mathbf{S}_{u}$ and define the matrices of inner products between $u$-th and $v$-th derivatives of the spline basis elements, $\mathbf{M}_{k, u}^{x}= \left(m_{ij}^x \right)_{i,j=-k+u}^g$,
$\mathbf{M}_{l, v}^{y}= \left(m_{ij}^y \right)_{i,j=-l+v}^h$, with
$$
m_{ij}^x:=\int\limits_a^b B_i^{k+1-u}(x)B_j^{k+1-u}(x) \, \mbox{d}x, \qquad
m_{ij}^y:=\int\limits_a^b B_i^{l+1-v}(y)B_j^{l+1-v}(y) \, \mbox{d}y.
$$
Finally, let $\mathbb{M}:=\mathbf{M}_{l,v}^{y} \otimes \mathbf{M}_{k,u}^{x}$.
Regarding the condition (\ref{cond2}), by using the well-known properties of the splines (see, e.g., \citep{deboor78,Schum}), it is possible to write
\begin{equation*}
\begin{split}
\int\limits_{a}^{b}\int\limits_{c}^{d} s_{kl}(x,y) \, \mbox{d}y \, \mbox{d}x & = \int\limits_{a}^{b}
\left[s_{k,l+1}(x,y) \right]_{c}^{d} \, \mbox{d}x =  \int\limits_{a}^{b} s_{k,l+1}(x,d)  \, \mbox{d}x
-  \int\limits_{a}^{b} s_{k,l+1}(x,c)  \, \mbox{d}x  = \\
& = \left[s_{k+1,l+1}(x,d) \right]_{a}^{b} - \left[s_{k+1,l+1}(x,c) \right]_{a}^{b} \, = \,  \\
& = s_{k+1,l+1}(b,d)-  s_{k+1,l+1}(a,d) - s_{k+1,l+1}(b,c) +  s_{k+1,l+1}(a,c)  \, = \, \\
& = c_{g,h} - c_{-k-1,h} - c_{g,-l-1} + c_{-k-1,-l-1}
\end{split}
\end{equation*}
using the notation
$
s_{k+1,l+1} (x,y) \, = \, \sum\limits_{i=-k-1}^{g} \sum\limits_{j=-l-1}^{h} c_{ij} \, B_{i}^{k+2}(x)\, B_{j}^{l+2}(y)
$
and the coincident additional knots (\ref{add_knots_x1}), (\ref{add_knots_y1}).
Accordingly, the condition (\ref{cond2}) is fulfilled if and only if
\begin{equation}\label{elim_c}
c_{-k-1,-l-1} \, = \, c_{-k-1,h} + c_{g,-l-1} - c_{g,h}.
\end{equation}
There is a useful relation between the $B$-spline coefficients of $s_{kl}(x,y)$ and $s_{k+1,l+1}(x,y)$,
which can be expressed in matrix notation as
$\mathbf{B} \, = \, \mathbf{D}_x \mathbf{K}_x \mathbf{C} \, \mathbf{K}_y^{\top}\mathbf{D}_y^{\top}$, where
$ \mathbf{C}=\left(c_{ij}\right)\in\mathbb{R}^{g+k+2,h+l+2}$,
$$
\mathbf{D}_x  =\, (k+1) \, diag\left\{ \dfrac{1}{\lambda_1-\lambda_{-k}},\ldots,\dfrac{1}{\lambda_{g+k+1}-\lambda_{g}} \right\},
$$
$$
\mathbf{D}_y  = \, (l+1) \, diag\left\{ \dfrac{1}{\mu_1-\mu_{-l}},\ldots,
\dfrac{1}{\mu_{h+l+1}-\mu_{h}} \right\},
$$
\begin{equation*}
\mathbf{K}_x \; = \; \left( \begin{array}{cccc}
 -1 & 1      &       & \\
    & \ddots & \ddots & \\
    &        & -1     & 1
        \end{array}\right)\in\mathbb{R}^{g+k+1,g+k+2},
 \qquad
 \mathbf{K}_y \; = \; \left( \begin{array}{cccc}
 -1 & 1      &       & \\
    & \ddots & \ddots & \\
    &        & -1     & 1
        \end{array}\right)\in\mathbb{R}^{h+l+1,h+l+2}.
\end{equation*}
By using notation $\mathbb{D} \, := \, \mathbf{D}_y \otimes \mathbf{D}_x$,
$\mathbb{K} \, := \, \mathbf{K}_y \otimes \mathbf{K}_x$ and
relation (\ref{elim_c}) to elide the coefficient $c_{-k-1,-l-1}$ we have
\begin{equation}\label{Ctilde}
  cs(\mathbf{B}) \, = \, \mathbb{D} \, \widetilde{\mathbb{K}} \, cs(\widetilde{\mathbf{C}}),
\end{equation}
where $cs(\widetilde{\mathbf{C}}) = (c_{-k,-l-1},\cdots,c_{g,-l-1},\cdots,c_{-k-1,h},\cdots,c_{g,h})^T$.
Having set this notation, we can now state explicitly the minimizer of \eqref{smoothR2a}, under the zero-integral constraint.

\begin{theorem}\label{thm:Juv-min}
The tensor smoothing spline $s_{kl}(x,y)\in{\cal S}_{k l}^{\Delta\lambda,
\Delta\mu}(\Omega)$, which minimizes the functional (\ref{smoothR2a}) under
the condition \eqref{cond2} is obtained as
\begin{equation}
\label{tss-thm}
	s_{kl}(x,y) \; = \; \mathbb{B}(x,y)  \,\mathbb{D} \, \widetilde{\mathbb{K}} \, cs(\widetilde{\mathbf{C^*}}),
\end{equation}
where
\begin{equation}\label{minB-thm}
cs(\widetilde{\mathbf{C^*}}) \; = \;
\left[ \widetilde{\mathbb{K}}^{\top} \mathbb{D}^{\top}
\left[(1-\alpha) \, \mathbb{S}^{\top} \mathbb{M} \, \mathbb{S} + \alpha \, \mathbb{B}^{\top} \, \mathbb{B}\right] \mathbb{D} \widetilde{\mathbb{K}}
\right]^{+} \alpha \,\widetilde{\mathbb{K}}^{\top}\mathbb{D}^{\top} \mathbb{B}^{\top} \, cs(\mathbf{F}),
\end{equation}
$\mathbf{F}=(f_{ij})$ and $cs(\mathbf{F})$ denotes its vectorized form.
\end{theorem}
As a by-product of the derivations leading to Theorem \ref{thm:Juv-min}, one indeed obtains the following result, which states the necessary and sufficient condition for bivariate splines to have zero integral (the proof is provided in the Supplementary Material).

\begin{theorem}
\label{splineR2}
For every  spline $s_{kl}(x,y) \in {\cal S}_{kl}^{\Delta\lambda,\Delta\mu}(\Omega)$, with the
representation
$s_{kl}\left(x,y\right)=\sum\limits_{i=-k}^{g} \sum\limits_{j=-l}^{h} b_{ij} \, B_{i}^{k+1} \left( x\right)\,  B_j^{l+1} \left(y\right)$,
the condition
$\iint \limits_{\Omega} s_{kl}(x,y) \, \mbox{d}x \mbox{d}y \, = \, 0$
is fulfilled if and only if
$$\sum\limits_{i=-k}^{g} \sum\limits_{j=-l}^{h} \; b_{ij} \left(\lambda_{i+k+1}-\lambda_i\right) \left(\mu_{j+l+1}-\mu_j\right) \; = \; 0.$$
\end{theorem}

The next important result is that, if using the proposed spline representation for the bivariate densities, the spline representations of the corresponding geometrical marginals in the clr space can be explicitly computed, and carry automatically the zero integral constraint, as stated in the next theorem.

\begin{theorem}\label{margTh}
Let $s_{kl}(x,y) \in {\cal S}_{kl}^{\Delta\lambda,\Delta\mu}(\Omega)$ such that
$\iint\limits_{\Omega} s_{kl}(x,y) \, \mbox{d}x \, \mbox{d}y \, = \, 0$ be given.
Let 
$s_k(x) \in {\cal S}_k ^{\Delta\lambda}[a,b]$,  $s_l(y) \in {\cal S}_l ^{\Delta\mu}[c,d]$
be defined as $s_k(x) = \int\limits_{c}^{d} s_{kl}(x,y) \mbox{d}y$, and $s_l(y) = \int\limits_{a}^{b} s_{kl}(x,y) \mbox{d}x$.
Then
$$
s_k(x)=\sum\limits_{i=-k}^g v_{i} B_{i}^{k+1}(x), \qquad s_l(y)=\sum\limits_{j=-l}^h u_{j} B_{j}^{l+1}(y)
$$
with
\begin{align*}
v_i &= \dfrac{b_{ih}}{t_h}+\cdots+\dfrac{b_{i,-l}}{t_{-l}}; \qquad t_{j'}= \dfrac{l+1}{\mu_{j'+l+1}-\mu_{j'}}, \quad j'=-l,\cdots,h,\\
u_j &= \dfrac{b_{jg}}{d_g}+\cdots+\dfrac{b_{j,-k}}{d_{-k}}; \qquad d_{j'}= \dfrac{k+1}{\lambda_{j'+k+1}-\lambda_{j'}}, \quad {j'}=-k,\cdots,g.
\end{align*}
Moreover, the splines $s_k(x),s_l(y)$ fulfil the zero-integral constraint, i.e.,
$$
\int\limits_{a}^{b} s_{k}(x) \mbox{d}x  \, = \, 0
\quad \mbox{and} \quad
\int\limits_{c}^{d} s_{l}(y) \mbox{d}y  \, = \, 0.
$$
\end{theorem}

The proof of Theorem \ref{margTh} is reported in the Supplementary Material. We finally introduce a spline representation for independent and interactive parts of the bivariate densities.

\begin{theorem}\label{sp_ind}
Let $s_{kl}(x,y) \in {\cal S}_{kl}^{\Delta\lambda,\Delta\mu}(\Omega)$,
$s_{kl}(x,y)=\sum\limits_{i=-k}^g \sum\limits_{j=-l}^h b_{ij} B_{i}^{k+1}(x) B_{j}^{l+1}(y)$, such that
$\iint\limits_{\Omega} s_{kl}(x,y) \, \mbox{d}x \, \mbox{d}y \, = \, 0$
be given.
Let $s_k(x) \in {\cal S}_k ^{\Delta\lambda}[a,b]$,  $s_l(y) \in {\cal S}_k ^{\Delta\mu}[c,d]$
be defined as
$$
s_k(x):= \int\limits_{c}^{d} s_{kl}(x,y) \mbox{d}y, \qquad
s_l(y):= \int\limits_{a}^{b} s_{kl}(x,y) \mbox{d}x
$$
with representation in the form
$$
s_k(x)=\sum\limits_{i=-k}^g v_{i} B_{i}^{k+1}(x), \qquad
s_l(y)=\sum\limits_{j=-l}^h u_{j} B_{j}^{l+1}(y).
$$
Then the independent part of the bivariate density $s_{kl}(x,y)$ admits the spline representation
$$
s_{kl}^{ind}(x,y) \, = \, s_k(x)\, + \, s_l(y) \, = \, \sum\limits_{i=-k}^g \sum\limits_{j=-l}^h (v_{ij}+u_{ij}) B_{i}^{k+1}(x) B_{j}^{l+1}(y),
$$
and the interactive part is expressed as the spline
$$
s_{kl}^{int}(x,y) \, = \, s_{kl}(x,y) \, - \, s_{kl}^{ind}(x,y) \, = \,
\sum\limits_{i=-k}^g \sum\limits_{j=-l}^h (b_{ij}-v_{ij}-u_{ij}) B_{i}^{k+1}(x) B_{j}^{l+1}(y),
$$
where
$
v_{ij}:=\dfrac{1}{d-c}v_i, \; u_{ij}:=\dfrac{1}{b-a}u_j \; \forall i,j.
$
\end{theorem}
The proof of Theorem \ref{sp_ind} is again reported to the Supplementary Material.

\smallskip
We remark that the results presented in this section form a computational cornerstone for the theoretical framework proposed in this work. Indeed, they not only allow for a complete characterization of bivariate densities though splines, but also for an explicit spline representation of the geometric marginals, as well as of the independence and interaction densities. Lastly, the spline representation enables one to compute the deviance and relative deviance from the interaction density $s_{kl}^{int}(x,y)$ simply as
$$
\Delta^2(s_{kl}^{int}(x,y)) \, = \,
\int\limits_a^b \int\limits_c^d  \left[s_{kl}^{int}(x,y)\right]^{2} \, \mbox{d}x\,\mbox{d}y.
$$
This result follows from an analogous development as that leading to the proof of Theorem \ref{thm:Juv-min} (namely, the derivation of $J_1$ by setting $u=0$, $v=0$, see the Supplementary Material for further details).

The next section showcases the application of the proposed methodology to a real dataset dealing with anthropometric measurements.


\section{An application to anthropometric densities}
\label{sec:anthropometric}


Periodic collection and reporting of anthropometric data such as body height and weight is essential to measure time trends in the prevalence of overweight and obesity at the population level. To this aim, a representative dataset of 4,436 Czech adolescents and young adults aged 15--31 years was collected as part of a large cross-sectional study (the reader may refer to \citep{gaba14,gaba16} for further details on the study). Participants to the study were selected on a volunteer basis among university students, staff and attendants to university open-house days and education exhibitions.
The sample sizes were, however, not distributed uniformly throughout the age intervals, mainly due to {a broader participation by} university students, see Table \ref{ssizes}.

\begin{table}[ht]
\caption{Sample sizes for age groups in anthropometric data.}\label{ssizes}
\begin{centering}
\begin{tabular}{l | c  c  c  c  c  c }
\hline
Age interval & $[15,16)$ & $[16,17)$ & $[17,18)$ & $[18,19)$ & $[19,20)$ & $[20,21)$ \\
Sample size & 95 & 126 & 234 & 492& 686 & 516 \\ \hline
Age interval& $[21,22)$ & $[22,23)$ & $[23,24)$ & $[24,25)$ & $[25,26)$ & $[26,27)$ \\
Sample size & 443& 450 & 385 & 318 & 220 & 155 \\ \hline
Age interval& $[27,28)$ & $[28,29)$ & $[29,30)$ & $[30,31)$&&\\
Sample size & 108 & 99 & 79 & 90&&\\ \hline
\end{tabular}

\end{centering}
\end{table}

Body height was measured with a precision of 0.1 cm by anthropometer P-375 (Trystom, Olomouc, Czech Republic) and body weight was measured using the InBody 720 device (Biospace Co., Ltd.; Seoul, Korea). Histogram data were then obtained from raw data, separately in each of $N=16$ age groups, i.e. $\left[15,16\right), \left[16,17\right),\ldots$, $\left[30,31\right)$. Note that the same age range was used also in \citep{machalova19}. The support of the marginal distribution of weights ($X$) and heights ($Y$) was set to the respective interval of observation, namely to $I_1=[40,100]$ for $X$ and $I_2=[155,195]$ for $Y$; the Sturge's rule was used to select the number of classes $K$,$L$ along the variable $X$,$Y$, respectively, in each of the histograms.
Possible (count) zeros in the histograms were imputed as advocated in \citep{martin15}; more precisely, a zero value in the class $k,l$ of a histogram was set to $2/3 n_{kl}$, where $n_{kl}$ stands for the number of observations within the class. For each age group, this led to the discrete representation $\{f_{kl}, k=1,\ldots,K,\, l=1,\ldots,L\}$ of the bivariate distributions, which were referred to the midpoints of the classes $\{\mathbf{t}_{kl}=(x_k,y_l), k=1,\ldots,K,\, l=1,\ldots,L\}$. The associated discrete bivariate clr transformations were computed as
\[\mathrm{clr}(f)_{kl}=\ln f_{kl} - \frac{1}{KL}\sum_{k=1}^K\sum_{l=1}^L \ln f_{kl}.\]
These clr transformations were smoothed by using the tensor product smoothing splines with zero integral introduced in Section \ref{sec:splines}, considering a rectangular domain $\Omega = [40; 100] \times [155; 195]$. For each age class, the following strategy was considered to set the parameters for the smoothing procedure. Quadratic smoothing splines were employed in each direction, setting the knots to equispaced sequences in both directions, with spacing of 15 kg along $X$ and 10 cm along $Y$. The order of derivative in the penalty term was set to $u=1$, $v=1$. The smoothing parameter $\alpha$ was determined by means of GCV errors over all sampled bivariate densities, resulting in $\alpha=0.0496$, see Figure~\ref{gcv}.
\begin{figure}[ht]
  \centering
  \includegraphics[width=0.4\textwidth]{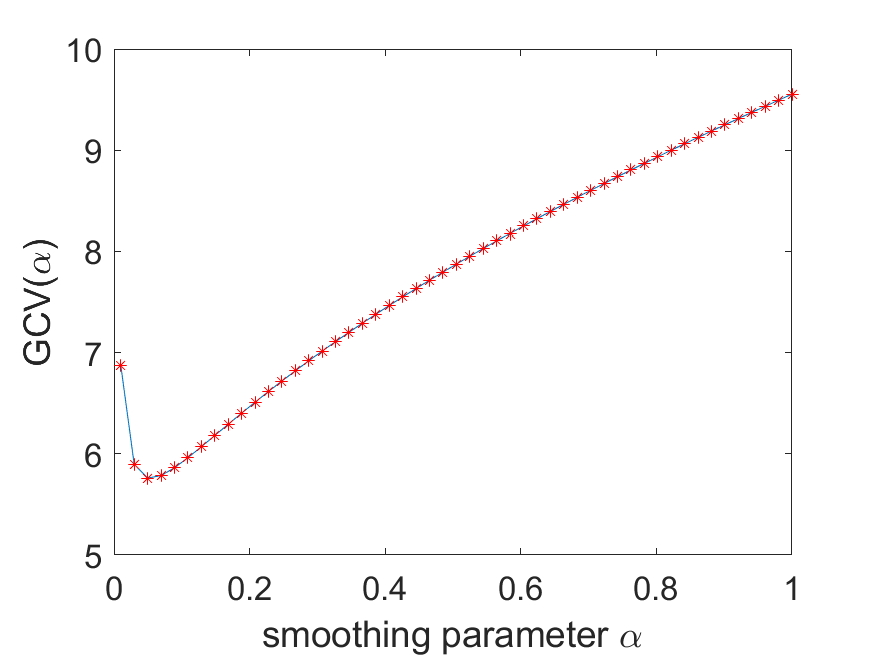}
  \caption{The mean GCV errors for the choice of the smoothing parameter $\alpha$ for fitting the anthropometric data by smoothing tensor splines with
	zero integral.}\label{gcv}
\end{figure}
The matrix $\mathbf{B}_t^*=(b_{ij}^t)$, $t=1, \ldots,16$, of coefficients for the smoothing spline $s_{kl}^t(x,y)$ with zero integral were finally obtained by \eqref{tss-thm}. A subset of the resulting clr densities, expressed with respect to the uniform measure on $\Omega$, are displayed in Figure \ref{fig-new}a. The corresponding densities (obtained by exponentiating the clr-transformations) are reported in Figure \ref{fig-new}b. The complete set of smoothed data and clr-transform is available in the Supplementary Material. One can clearly see the bimodal character of the densities, which is probably due to the presence of both males and females in the sample. In addition, some dependence between heights and weights is apparent, without a substantial difference between the two modes, for most densities.

\begin{figure}[ht]
		\centering
		\includegraphics[width=\textwidth]{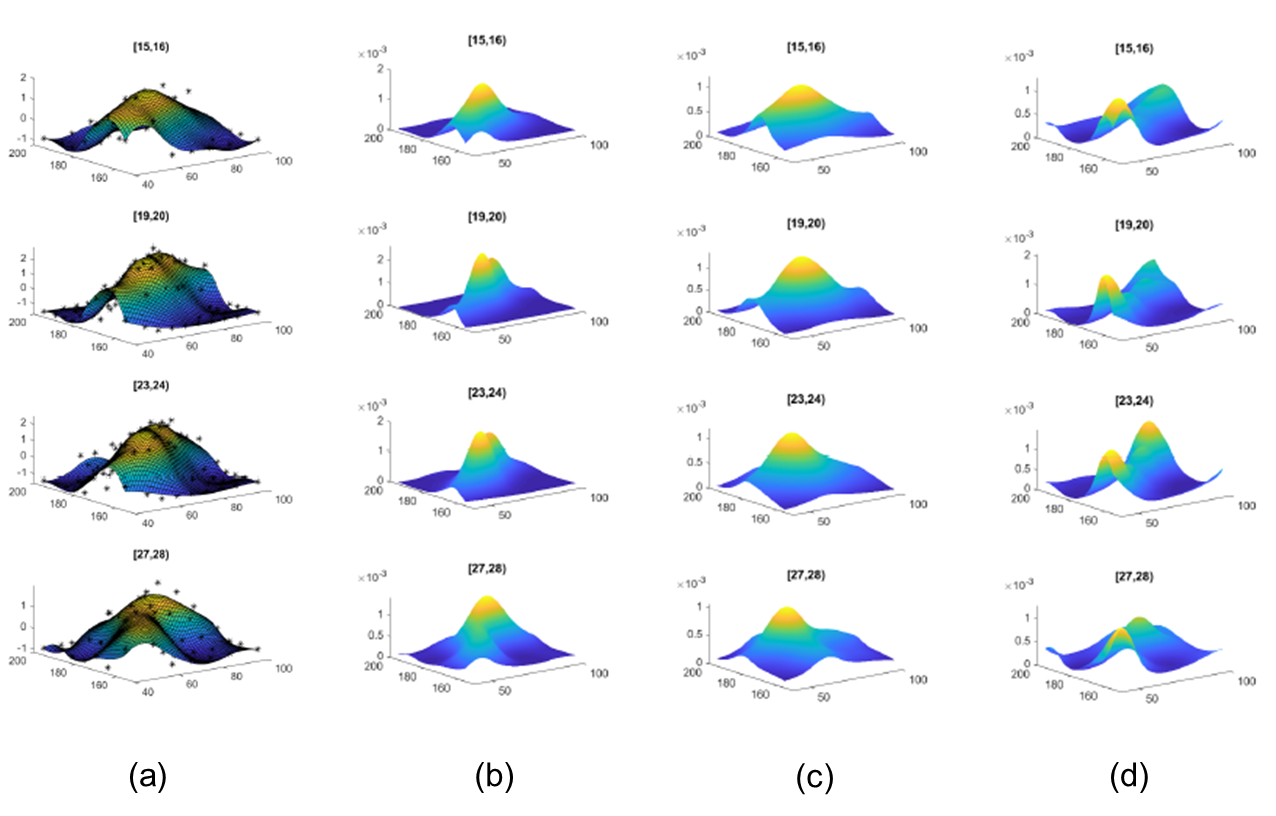}
  \caption{Anthropometric data: part of the smoothed data and their decomposition in independent and interactive part. (a) Smoothed clr-densities; symbols represent the discrete clr transformation at mid-points of histogram classes, the smooth surfaces represent the smoothed PDFs reconstructed via smoothing splines. (b) Smoothed densities. (c) Independent part. (d) Interactive part. Computations refer to the uniform reference measure.}\label{fig-new}
\end{figure}

From the smoothed data, the decomposition of the bivariate densities into their independent and interactive parts was computed using the results detailed in Section \ref{sec:splines}, based on the corresponding B-spline coefficients (see Theorem \ref{sp_ind}). The independent and interactive parts of the densities in Figure \ref{fig-new}b are reported in Figures \ref{fig-new}c and  \ref{fig-new}d.
It is interesting to observe that the bimodal character of the densities almost disappears in the independent densities, as this feature is mostly captured by the interaction densities. Apart of that, it is obvious that upper/lower/combined extreme values in the variables heights and weights have a relevant contribution on the dependence between the variables in the data set. This can be seen on the uplifted values appearing in corners of the majority of the interaction densities
(Figure \ref{fig-new}d).

We now aim to investigate whether the dependence between height and weight changes with ageing of the population. For this purpose, the simplicial deviances $\Delta^2(f_i)$ and the relative simplicial deviances $R^2(f_i)$, $i=1,\ldots,16$, were computed
as described in Section \ref{sec:splines},
see Table~\ref{tabsd}. Inspection of Table \ref{tabsd} suggests that simplicial deviances are clearly influenced by the sample sizes in the age intervals, yielding higher values of $\Delta^2(f_i)$ between ages 18 and 24, due more local effects resulting from the smoothing of histograms with more classes. These effects are filtered out in the relative simplicial deviances, whose time series is reported in Figure \ref{corr} (upper figure).

\begin{center}
\begin{table}[ht]
\caption{Values of norms of the bivariate densities and their decomposition together with the derived coefficients for all age groups in anthropometric data.}\label{tabsd}
\begin{tabular}{c | c || c | c | c | r | c}
$i$ & age group & $||f_i||_{\mathcal{B}^2(\mathsf{P})}$ & $||f^i_{\mathrm{ind}}||_{\mathcal{B}^2(\mathsf{P})}$ & $||f^i_{\mathrm{int}}||_{\mathcal{B}^2(\mathsf{P})}$ & $\Delta^2 (f_i)$ & $R^2(f_i)$\\
\hline
\hline
1  & [15,16) & 44.786 & 34.853 & 28.125 &  791.007 & 0.394 \\
2  & [16,17) & 40.437 & 31.518 & 25.333 &  641.772 & 0.392 \\
3  & [17,18) & 49.403 & 32.794 & 36.948 & 1365.180 & 0.559 \\
4  & [18,19) & 55.132 & 34.132 & 43.296 & 1874.570 & 0.617 \\
5  & [19,20) & 63.097 & 40.634 & 48.271 & 2330.051 & 0.585 \\
6  & [20,21) & 58.650 & 36.198 & 46.147 & 2129.520 & 0.619 \\
7  & [21,22) & 54.236 & 32.744 & 43.236 & 1869.361 & 0.636 \\
8  & [22,23) & 54.305 & 34.407 & 42.014 & 1765.162 & 0.599 \\
9  & [23,24) & 52.934 & 31.971 & 42.189 & 1779.876 & 0.635 \\
10 & [24,25) & 52.647 & 29.572 & 43.557 & 1897.188 & 0.684 \\
11 & [25,26) & 45.313 & 26.604 & 36.680 & 1345.455 & 0.655 \\
12 & [26,27) & 40.893 & 24.782 & 32.528 & 1058.088 & 0.633 \\
13 & [27,28) & 34.739 & 21.963 & 26.915 &  724.405 & 0.600 \\
14 & [28,29) & 33.712 & 19.344 & 27.610 &  762.289 & 0.671 \\
15 & [29,30) & 31.862 & 20.223 & 24.621 &  606.206 & 0.597 \\
16 & [30,31) & 30.849 & 17.031 & 25.722 &  661.618 & 0.695 \\
\end{tabular}
\end{table}
\end{center}

\begin{figure}[ht]
  \centering
  \includegraphics[width=0.5\textwidth]{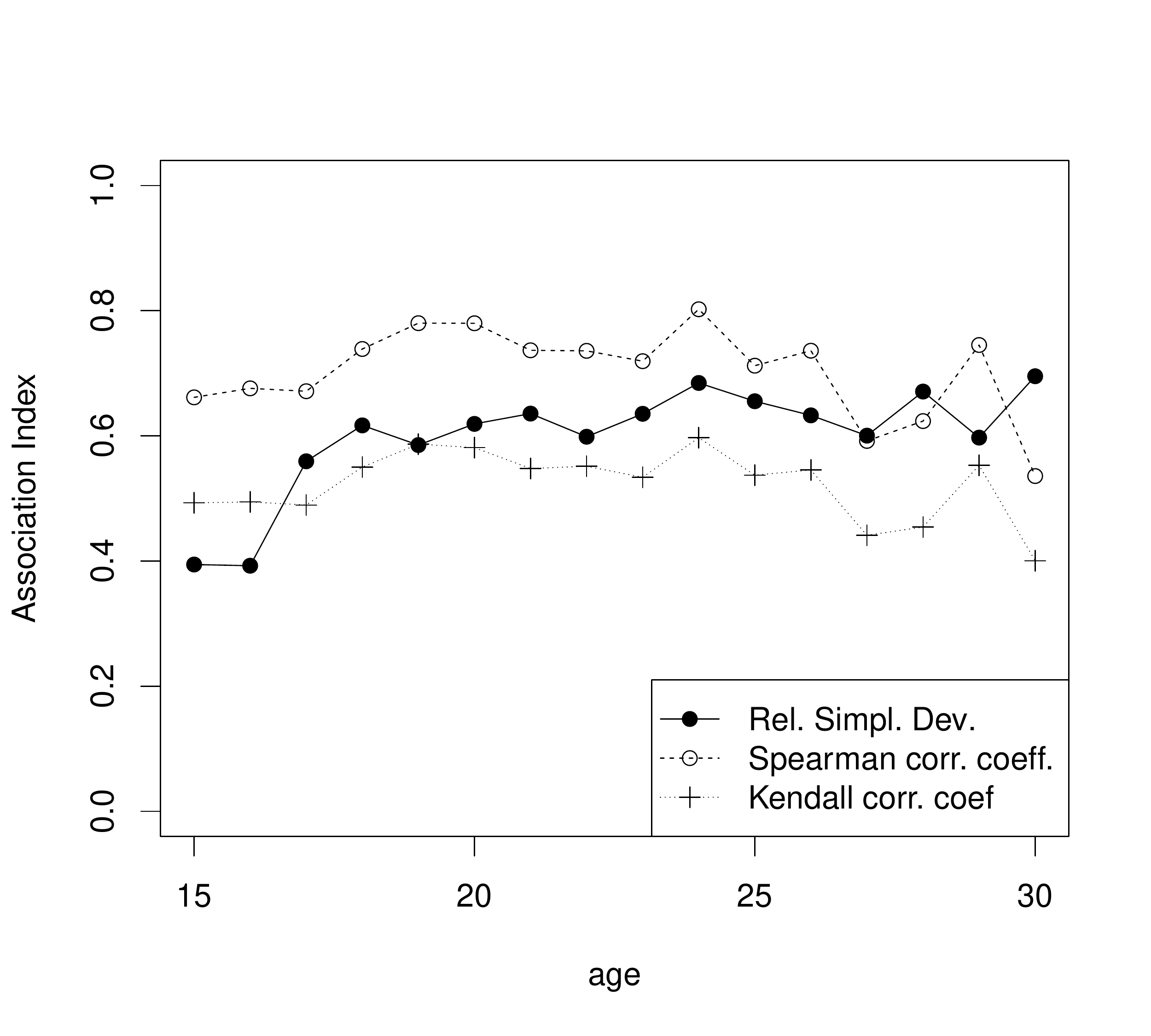}\\
  \caption{Time series of the relative simplicial deviances (solid line), the Spearman (dashed line) and the
Kendall (dotted line) correlation coefficients for anthropometric data.}\label{corr}
\end{figure}

A stationarity check performed with the Kwiatkowski-Phillips-Schmidt-Shin (KPSS) test \cite{kwiatkowski92} yields a narrow rejection of the stationarity assumption ($p$-value $p=0.0482$). As a consequence, (slightly) non-stationary effects emerge in the time series of relative simplicial deviances. A major such effect can be observed at the beginning of the time series; here the relative simplicial deviance slightly increases and gets stabilized around the nineteenth year. This development can be easily explained by pubertal and postpubertal changes in height and weight, which still occur in the mentioned time period.

It is interesting to compare the relative simplicial deviance with the well-known Spearman and Kendall correlation coefficients, which both are closely connected to the copula theory \citep{nelsen06}. Their time series are displayed in Figure \ref{corr} (dashed and dotted lines, respectively) and they look pretty similar. The effect of (post)pubertal changes is no more visible here; we can rather observe a slightly decreasing trend from around 24 years, which, interestingly, corresponds to border age of the ``Youth'' age group according to World Health Organization \cite{who}. This would indicate that since then the strength of the monotonic dependence between height and weight slightly weakens. This result reflects the fact that unlike both the mentioned correlation coefficients, the relative simplicial deviance captures the whole ``mass'' of the interactions between height and weight distributions, thus also including effects like possible tail dependence. Note that a similar idea of using a norm for development of a dependence measure is followed in \cite{tran15} with the Sobolev metric, thought there with the aim to capture rather solely monotone dependence.

We now further investigate the non-stationarity effects being observed for the time series of relative simplicial deviances. For this purpose, we formulate a compositional regression model with functional response \citep{Tal} and scalar regressors (i.e., the time $t$ of observation) which results in the linear model
\begin{equation}
\label{regr}
f_{\mathrm{int},i}=\beta_0\oplus\beta_1\odot t_i\oplus\varepsilon_i,
\end{equation}
for $i=1,\ldots,16$, with $\beta_0,\beta_1$ unknowns coefficients in $\mathcal{B}^2(\mathsf{P})$ and $\varepsilon_i$ a zero-mean random error. Note that, by linearity, the properties of $f_{\mathrm{int},i}$ (as stated in Theorems \ref{thm:ortho}, \ref{thm:ortho-mar} and \ref{thm:int-neutral}) are inherited by $\beta_0,\beta_1$, as $\mathbb{E}[f_{\mathrm{int}}] = \beta_0\oplus\beta_1\odot t $. Applying the clr transformation (\ref{biclr}) to both sides of the model (\ref{regr}) yields
\begin{equation}
\label{regr-clr}
f_{\mathrm{int},i}^c(x,y)=\beta_0^c(x,y)+\beta_1^c(x,y) t_i+\varepsilon_i^c(x,y),\quad i=1,\ldots,16, (x,y)\in I_1\times I_2.
\end{equation}
For the estimation of the functional regression parameters using the least squares criterion the smoothing tensor spline coefficients (Theorem \ref{sp_ind}) can be utilized. Similarly as in \cite{Tal}, the spline coefficients of the regression estimates (clr transformed densities) $\widehat{\beta}_0^c(x,y)$ and $\widehat{\beta}_1^c(x,y)$ fulfill the condition from Theorem \ref{splineR2}. The clr-transformation of the estimated parameters $\widehat{\beta}_0, \widehat{\beta}_1$ are reported in Figure \ref{fig:Beta}. A permutation test on the global significance of the parameter $\beta_1$ -- run using a Freedman and Lane scheme \citep{Freedman,Pini} with test statistics $T^2=\|\widehat{\beta}_1\|^2_{\mathcal{B}^2(\mathsf{P})}$ -- confirms the statistical significance of this parameter (p-value 0.032), suggesting that the time variation in the interaction between the random variables is indeed relevant. In the light of the shape of $\widehat{\beta}_1^c$ (Figure \ref{fig:Beta}) one can conclude that, along time, the interaction between weights and heights tends to get more concentrated at medium-high values or low values of the height, whereas it deflates for medium-low values of the height. A less pronounced time variation is instead observed for different values of the weight, suggesting that the highest variability in the dependence between these variables is indeed observed across values of the other variable.

\begin{figure}[ht]
  \centering
  \includegraphics[width=0.3\textwidth]{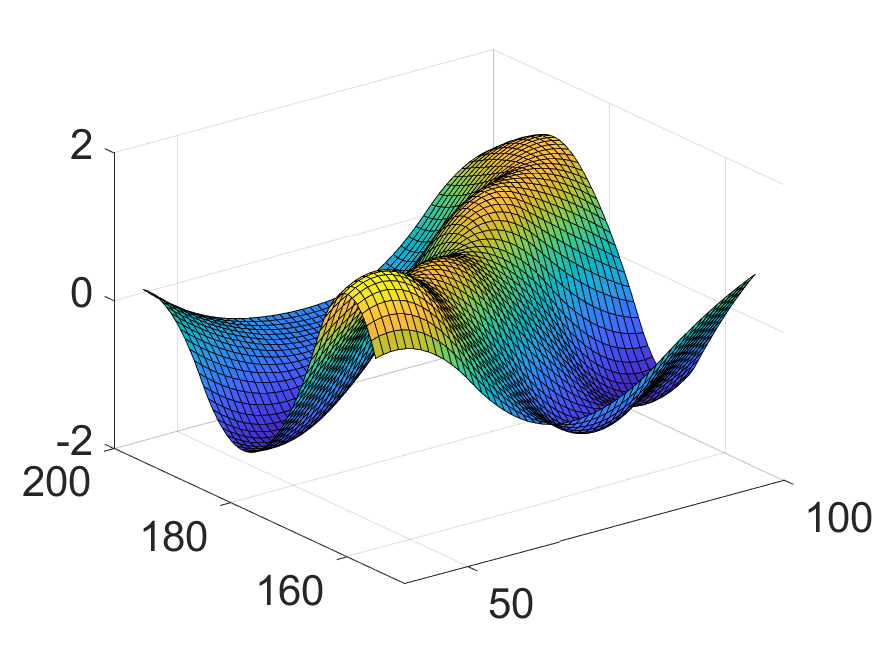}\includegraphics[width=0.3\textwidth]{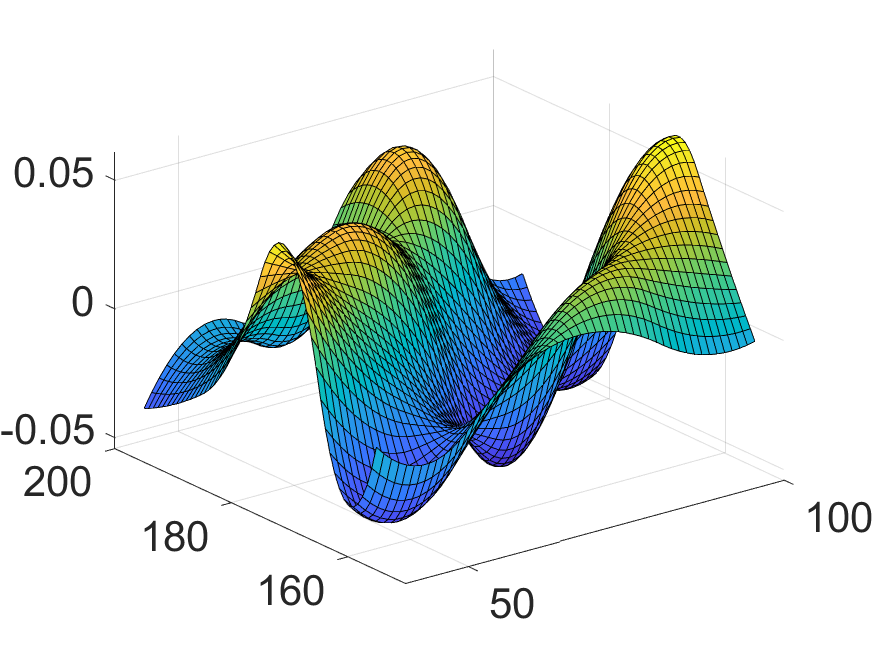}\\
  \caption{Clr-transformation of the estimated parameters $\widehat{\beta}_0^c(x,y)$ and $\widehat{\beta}_1^c(x,y)$.}\label{fig:Beta}
\end{figure}


\section{Conclusions}
\label{sec:conc}


The Bayes spaces methodology provides a robust and flexible framework for modelling data with relative character, including measures, probability density functions as well as compositional data. It can serve for many different purposes, from geometrical representation of the Bayes theorem to functional data analysis of a sample of densities. In this paper, its potential was further extended to bivariate density functions. Their decomposition into independent and interactive parts has a solid geometrical basis and allows for an appealing probabilistic interpretation if a normalized reference measure is used. This opens new perspectives for both further generalization to multivariate densities as well as to dependence modelling, with the aim to provide an alternative viewpoint than that offered by the widely-used copula theory. Note that the Bayes space theory is built for general types of positive measures (not necessarily absolutely continuous); we here foresee clear perspectives of development for a general Bayes space approach for distributions, in contexts and assumptions even closer to those of the well-established theory of copulas.
Other important envisioned impacts of this work are worth to be mentioned. For instance, the spline smoothing here developed may be used for a non-parametric estimate of PDFs, directly allowing for further data processing in the view of FDA, although at the expense of possible lower convergence rate than empirical distribution functions or empirical copulas constructed using ranks. In fact, from the application viewpoint and in the light of the promising theoretical properties presented in Section \ref{sec:decomp}, the Bayes space approach could be used to develop novel FDA methods for bivariate densities, and provide a broader statistical framework to pioneering applications as those developed in \cite{iacopini19}. An instance of this has been shown in Section \ref{sec:anthropometric}, where a linear regression model for interaction densities has been formulated to further investigate the variability of the interaction between two random variables along time. Such linear modeling would not be easy in other settings, based on non-linear and non-orthogonal relations between independent and interaction parts.
More in general, in this first work, we outlined a number of novel views allowed by the proposed spline representation, with particular reference to the statistical analysis of bivariate densities in Bayes spaces. In fact, we here envision a great potential of this framework, which can be used to provide a mathematical setting for the statistical processing of samples of bivariate densities in varied contexts. Note that distributional datasets are becoming increasingly available in the applications, as these could result from aggregation of massive data coming from large-scale studies or automated collection of data. Despite this, the statistical methods available for their analysis (particularly in the multivariate case) are still limited. Depending on whether such densities form a random sample, regionalized observations or time series, appropriate methods of FDA far beyond those explicitly mentioned in this work can be built, precisely grounding upon the presented theory and the associated spline representations.


\section*{Acknowledgements}

The first and second author were supported by Czech Science Foundation (GACR), GA19-01768S.





\bibliographystyle{plainnat} 
\bibliography{ref}


\newpage
{\bf\Large Bivariate densities in Bayes spaces: orthogonal decomposition and spline representation-
supplementary material}


\section*{Supplement A: Spline representation of univariate clr transformed densities}


In this supplementary section the terminology and basics for the spline representation of clr transformed univariate densities as $L^2$ functions with zero integral are recalled.
Let the sequence of knots $\Delta\lambda \, := \, \left\{ \lambda_i \right\}_{i=0}^{g+1}$, $\lambda_{0}=a<\lambda_{1}<\ldots<\lambda_{g}<b=\lambda_{g+1}$ be given.
The symbol ${\cal S}_{k}^{\Delta\lambda}[a,b]$ denotes the vector space of polynomial splines of degree $k>0$, defined on a finite interval $[a,b]$ with the sequence
of knots $\Delta\lambda$. It is known that $\dim\left({\cal S}_{k}^{\Delta\lambda}[a,b]\right)=g+k+1$.
Then every spline $s_{k}(x)\in{\cal S}_{k}^{\Delta\lambda}[a,b]$ has an unique representation
\begin{equation*}
    s_{k}\left(x\right)=\sum_{i=-k}^{g}b_{i}B_{i}^{k+1}\left(x\right).
\end{equation*}

For generalization of splines to the bivariate density case the following theorem, which was published in \cite{Tal}, is of paramount importance.
\begin{theorem}
\label{spline}
For a spline $s_{k}(x)\in{\cal S}_{k}^{\Delta\lambda}[a,b]$, $s_{k}\left(x\right)=\sum\limits_{i=-k}^{g}b_{i}B_{i}^{k+1}\left(x\right)$,
the condition $\int\limits_{a}^{b}s_{k}(x)\,\mbox{d}x=0$ is fulfilled if and only if
$\sum\limits_{i=-k}^{g}\;b_i\left(\lambda_{i+k+1}-\lambda_i\right)\;=\;0.$
\end{theorem}

\begin{proof}
From the spline theory it is known that
$\int s_k(x) \, \mbox{d}x \, = \, s_{k+1}(x)$.
If the notation
$s_{k}(x)  = \sum\limits_{i=-k}^{g}b_{i}B_{i}^{k+1}\left(x\right) $ is used,
$s_{k+1}(x)  =  \sum\limits_{i=-k-1}^{g} c_{i} B_{i}^{k+2}\left(x\right) $,
there is known the relationship between their $B$-spline coefficients in the form
$$
b_i \; = \; (k+1) \dfrac{c_i-c_{i-1}}{\lambda_{i+k+1}-\lambda_i},
\quad \forall i=-k,\ldots,g.
$$
Thus the coefficients $c_i$ can be expressed as
$$
c_i \; = \; c_{i-1}+ \dfrac{b_i}{d_i}, \quad \forall i=-k,\ldots,g
$$
with $d_i=\dfrac{k+1}{\lambda_{i+k+1}-\lambda_i}$
and it means that
$$
c_g \, = \,  \dfrac{b_g}{d_g} \, + \, \cdots \, + \, \dfrac{b_{-k}}{d_{-k}} \, + \, c_{-k-1}.
$$
According to the coincident additional knots, see \cite{Mach} for details, it holds
\begin{equation}\label{int_c}
  \int\limits_{a}^{b} s_{k}(x) \,\mbox{d}x \; = \; \left[ s_{k+1}(x) \right]_a^b \; = \;
  s_{k+1}(b) - s_{k+1}(a) \; = \; c_{g} - c_{-k-1},
\end{equation}
and it is obvious that
$$
0 \, = \, \int\limits_{a}^{b} s_{k}(x) \, \mbox{d}x \quad \Leftrightarrow \quad c_g = c_{-k-1} \quad
\Leftrightarrow \quad \dfrac{b_g}{d_g} \, + \, \cdots \, + \, \dfrac{b_{-k}}{d_{-k}} \, = \, 0.
$$
Finally, the definition of $d_i$ implies that the following sequence of equivalences can be formulated,
$$
0 \, = \, \int\limits_{a}^{b} s_{k}(x) \, \mbox{d}x \quad  \Leftrightarrow \quad
\sum\limits_{i=-k}^{g} \dfrac{b_i}{d_i} \, = \, 0 \quad  \Leftrightarrow \quad
\sum\limits_{i=-k}^{g}  b_i\left( \lambda_{i+k+1}-\lambda_i \right) \, = \, 0.
$$
\end{proof}

\noindent
{\bf Algorithm}\\
The algorithm to find a spline $s_{k}(x)\in{\cal S}_{k}^{\Delta\lambda}[a,b]$ with zero integral, i.e., the respective
vector $\mathbf{b}=(b_{-k}, \cdots, b_g)^{\top}$, can be summarized as follows:\\
{\bf 1.} choose $g+k$ arbitrary $B$-spline coefficients $b_i\in\mathbb{R}$, $i=-k\ldots,j-1,j+1,\ldots,g$,\\
{\bf 2.} compute
$$
b_j \; = \; \dfrac{-1}{\lambda_{j+k+1}-\lambda_j} \; \sum\limits_{i=-k\atop i\neq j}^{g}\;b_i\left(\lambda_{i+k+1}-\lambda_i\right).
$$


\section*{Supplement B: Proofs}


\begin{proof}[Proof of Theorem \ref{thm:ortho}]
The clr transformation of the independence density $f_{\mathrm{ind}}(x,y)$ can
be written as
\begin{equation}
\label{clrind}
f^c_{\mathrm{ind}}(x,y)=\ln [f_X(x)f_Y(y)] - \frac{1}{\mathsf{P}(\Omega)} \int\limits_{\Omega_X} \int\limits_{\Omega_Y}\ln [f_X(x)f_Y(y)]\,d\mathsf{P}_X d\mathsf{P}_Y.
\end{equation}
This is invariant under rescaling of the product $f_X(x)f_Y(y)$. By choosing the following representations of $f_X(x)$ and $f_Y(y)$,
$$f_X(x)=\exp[f_X^c(x)],\quad f_Y(y)=\exp[f_Y^c(y)],$$
the second term in (\ref{clrind}) equals zero. Thus (\ref{clrind}) can be rewritten as
$$f^c_{\mathrm{ind}}(x,y)=\ln\{\exp[f_X^c(x)+f_Y^c(y)]\}=f_X^c(x)+f_Y^c(y).$$
For the sake of simplicity in notation, arguments are hereafter omitted. Consider
$$f^c_{\mathrm{int}}=f^c-f^c_{\mathrm{ind}}=f^c-f^c_X-f^c_Y,$$
then
\begin{align*}
\langle f^c_{\mathrm{int}},f^c_{\mathrm{ind}} \rangle_{L_0^2(\mathsf{P})}&=\langle f^c-f^c_X-f^c_Y,f^c_X+f^c_Y \rangle_{L_0^2(\mathsf{P})}=\\
& = \langle f^c,f_X^c\rangle + \langle f^c,f_Y^c\rangle_{L_0^2(\mathsf{P})} -||f_X^c||_{L_0^2(\mathsf{P})}^2-||f_Y^c||_{L_0^2(\mathsf{P})}^2- 2\langle f_X^c,f_Y^c\rangle_{L_0^2(\mathsf{P})}.
\end{align*}
For the first scalar product one has
\begin{align*}
\langle f^c,f_X^c\rangle_{L_0^2(\mathsf{P})}& = \int\limits_{\Omega_X} \int\limits_{\Omega_Y}f^c(x,y) f^c_X(x)\,d\mathsf{P}_X d\mathsf{P}_Y=\\
&= \int\limits_{\Omega_X}f^c_X(x)\int\limits_{\Omega_Y}f_c(x,y)\,d\mathsf{P}_X d\mathsf{P}_Y=
\int\limits_{\Omega_X}[f^c_X(x)]^2\,d\mathsf{P}_X=||f_X^c||_{L_0^2(\mathsf{P})}^2,
\end{align*}
similarly also $\langle f^c,f_Y^c\rangle_{L_0^2(\mathsf{P})}=||f_Y^c||_{L_0^2(\mathsf{P})}^2$.
Finally,
$$\langle f_X^c,f_Y^c\rangle_{L_0^2(\mathsf{P})} = \int\limits_{\Omega_X} \int\limits_{\Omega_Y}f_X^c(x) f^c_Y(y)\,d\mathsf{P}_X d\mathsf{P}_Y = \int\limits_{\Omega_X}f_X^c(x)\,d\mathsf{P}_X \cdot \int\limits_{\Omega_Y}f^c_Y(y)\,d\mathsf{P}_Y=0, $$
which completes the proof.
\end{proof}

\begin{proof}[Proof of Theorem \ref{thm:ind-amar}]
In case of independence, one may decompose a bivariate density as the product of its arithmetic marginals as $f(x,y)=f_{X,a}(x)f_{Y,a}(y)$. In Bayes spaces, this is reformulated as in \eqref{inddecomp}. Call $f_X^c(x)$, $f_Y^c(y)$ the clr-representation of the marginals, i.e., $f_{X,a}(x)=\exp[f_{X,a}^c(x)]$ and similarly $f_{Y,a}(y)=\exp[f_{Y,a}^c(y)]$. Using \eqref{inddecomp}, one may build the independent component as  $f^c_{\mathrm{ind}}(x,y)=f_{X,a}^c(x)+f_{Y,a}^c(y)$, which clearly coincides with $f$ itself.
The clr representation of the geometric $X$-marginal is derived -- by definition \eqref{xclrmar} -- as
$$\int_{\Omega_Y}f^c_{\mathrm{ind}}(x,y)\,d\mathsf{P}_Y=\int_{\Omega_Y}f_{X,a}^c(x)\,d\mathsf{P}_Y=\mathsf{P}_Y(\Omega_Y)f_{X,a}^c(x).$$
By considering that $\mathsf{P}_Y(\Omega_Y)=1$, the geometric $X$-marginal is obtained by applying the exponential as $f_X(x)=\exp[f_{X,a}^c(x)]$, i.e., it coincides with the arithmetic marginal $f_{X,a}(x)$. The case of $Y$-marginals would be proven analogously.
\end{proof}

\begin{proof}[Proof of Theorem \ref{thm:ortho-mar}]
The orthogonality of the marginals is easy to be proven in the clr space.
Specifically,
\begin{align*}
\langle f_X^c,f_Y^c\rangle_{L_0^2(\mathsf{P})} & =\left\langle \int_{\Omega_X}f^c(x,y)\,d\mathsf{P}_X,\int_{\Omega_Y}f^c(x,y)\,d\mathsf{P}_Y\right\rangle_{L_0^2(\mathsf{P})}=\\
& = \int_{\Omega_X}\int_{\Omega_Y}\left[\int_{\Omega_X}f^c(x,y)\,d\mathsf{P}_X\right]\left[\int_{\Omega_Y}f^c(x,y)\,d\mathsf{P}_Y\right]d\mathsf{P}_Xd\mathsf{P}_Y=\\
&= \int_{\Omega_Y}\left[\int_{\Omega_X}f^c(x,y)\,d\mathsf{P}_X\right]d\mathsf{P}_Y\cdot\int_{\Omega_X}\left[\int_{\Omega_Y}f^c(x,y)\,d\mathsf{P}_Y\right]d\mathsf{P}_X=0
\end{align*}
from the fact that $f_X^c\in L_0^2(\Omega_X)$ and $f_Y^c\in L_0^2(\Omega_Y)$.
In the next step the orthogonality between $f_{\mathrm{int}}\equiv f_{\mathrm{int}}(x,y)$ and the $X$-marginal is proven. Using the first part of this theorem and the relation
$\langle f^c,f_X^c\rangle_{L_0^2(\mathsf{P})}=||f_X^c||^2_{L_0^2(\mathsf{P})}$ from the proof of Theorem \ref{thm:ortho} it holds
\begin{align*}
\langle f_{\mathrm{int}}^c,f_X^c\rangle_{L_0^2(\mathsf{P})} & =\langle f^c- f_{\mathrm{ind}}^c,f_X^c\rangle_{L_0^2(\mathsf{P})}=\langle f^c- f_X^c-f_Y^c,f_X^c\rangle_{L_0^2(\mathsf{P})}= \\
 &= \langle f^c,f_X^c\rangle_{L_0^2(\mathsf{P})}-||f_X^c||_{L_0^2(\mathsf{P})}^2-\langle f_X^c,f_Y^c\rangle_{L_0^2(\mathsf{P})}=\\
 & =||f_X^c||_{L_0^2(\mathsf{P})}^2-||f_X^c||_{L_0^2(\mathsf{P})}^2=0.
\end{align*}
\end{proof}

\begin{proof}[Proof of Theorem \ref{thm:int-neutral}]
Equation \eqref{eq:thm3} can be equivalently stated in terms of the clr marginals as
\begin{equation}\label{eq:thm3-proof1}
f^c + f^c_{\mathrm{int},X} = f; \quad f^c +  f^c_{\mathrm{int},Y} = f^c,\end{equation}
$f^c$ denoting the clr  transformation of $f$.
In this case, one has
\begin{align*}
f^c + f^c_{\mathrm{int},X} &= f^c + \int_{\Omega_X} f^c_{\mathrm{int}} d\mathsf{P}_X =\\
 &= f^c + \int_{\Omega_X} f^{c}\, d\mathsf{P}_X - \int_{\Omega_X} f^{c}_X d\mathsf{P}_X - \int_{\Omega_X} f^{c}_Y d\mathsf{P}_X = \\
 &= f^c + f^c_Y - f^c_Y \cdot\mathsf{P}_X(\Omega_X) = f^c,
\end{align*}
where the last equality holds true if the measure $\mathsf{P}_X(\Omega_X)$ is normalized.
With analogous argument, the same equality is proven for $f_{\mathrm{int},Y}^c$.
\end{proof}

\begin{proof}[Proof of Theorem \ref{thm:marg-pert}]
From (\ref{inddecomp}) and the expression $g_{\mathrm{ind}}=(g_X\oplus f_X)\oplus (g_Y\oplus f_Y)$ it follows that $g_{\mathrm{ind}}$ is an independence density of $g$. Therefore
$$g_{\mathrm{int}}=g\ominus g_{\mathrm{ind}}=f\ominus f_{\mathrm{ind}}=f_{\mathrm{int}}.$$
\end{proof}

\begin{proof}[Proof of Theorem \ref{thm:Juv-min}]
Let the first term in (\ref{smoothR2a}) be denoted as
\begin{equation}
\label{j2}
J_1 \; = \; \alpha\sum\limits_{i=1}^{n} \sum\limits_{j=1}^m \, \left[f_{i j}-s_{k l}(x_{i},y_j)\right]^{2}
\end{equation}
and the second one as
\begin{equation}
\label{j1}
J_2 \; = \; (1-\alpha)\iint\limits_{\Omega}\left[s_{k l}^{(u,v)}(x,y)\right]^{2}\, \mbox{d}x\,\mbox{d}y
\end{equation}
We can express the functional $J_1$ from (\ref{j2}) in matrix notation as
\begin{equation*}
	\begin{split}
	J_1 & = \, \alpha \sum\limits_{i=1}^{n} \sum\limits_{j=1}^m \, \left[f_{i j}-s_{k l}(x_{i},y_j)\right]^{2} \, = \,
            \alpha \left[ cs(\mathbf{F}) - \mathbb{B} \, cs(\mathbf{B})\right]^{\top} \left[ cs(\mathbf{F}) - \mathbb{B} \, cs(\mathbf{B}) \right] \, = \\
	    & = \, \alpha \left(cs(\mathbf{F})\right)^{\top} cs(\mathbf{F}) - 2 \alpha \left(cs(\mathbf{B})\right)^{\top} \mathbb{B}^{\top} \, cs(\mathbf{F}) +
            \alpha \left(cs(\mathbf{B})\right)^{\top} \mathbb{B}^{\top} \, \mathbb{B} \, cs(\mathbf{B}),
	    \end{split}
\end{equation*}
where $\mathbf{F}=(f_{i j})$, $\mathbb{B} \, := \, \mathbf{B}_{l+1}(\mathbf{y}) \otimes \mathbf{B}_{k+1}(\mathbf{x})$, $\mathbf{y}=(y_1,\cdots,y_m)$,
$\mathbf{x}=(x_1,\cdots,x_n)$.
Now we consider the derivative of the spline. Similarly as in case of one-
dimensional splines, \cite{Mach,Mach1}, the derivative can be expressed by
using (\ref{spline-biv}), (\ref{spline_tensor}) as
\begin{equation}
\label{derivsplines}
\begin{split}
s_{k l}^{(u,v)}(x,y) & = \, \frac{\partial^{u}}{\partial x^{u}} \, \frac{\partial^{v}}{\partial y^{v}} \,
\sum\limits_{i=-k}^{g} \sum\limits_{j=-l}^{h} \, b_{ij} \, B_i^{k+1}(x) \, B_j^{l+1}(y) = \\
& = \, \frac{\partial^{u}}{\partial x^{v}} \, \frac{\partial^{v}}{\partial y^{v}} \,
 \left( \mathbf{B}_{l+1}(y) \otimes \mathbf{B}_{k+1}(x) \right) \,  cs(\mathbf{B}) =\\
 & = \left[ \mathbf{B}_{l+1-v}(y)\mathbf{S}_{v} \otimes \mathbf{B}_{k+1-u}(x)\mathbf{S}_{u} \right] \, cs(\mathbf{B}).
\end{split}
\end{equation}
With respect to the properties of tensor product, and using the notation
$\mathbb{B}^{u,v}(x,y) := \mathbf{B}_{l+1-v}(y) \otimes \mathbf{B}_{k+1-u}(x)$, the derivative given in  (\ref{derivsplines}) can be reformulated as
$
s_{k l}^{(u,v)}(x,y) = \mathbb{B}^{u,v}(x,y) \, \mathbb{S}  \, cs(\mathbf{B}).
$
Note that the flexibility in the choice of the orders $u,\, v$ in the derivatives $s_{k l}^{(u,v)}(x,y)$ can be considered as an element of innovation with respect to the classical tensor smoothing spline approach \cite{dierckx93}.
Then the functional $J_2$ from (\ref{j1}) can be rewritten as
\begin{equation*}
	\begin{split}
		J_2 & = \, (1-\alpha)\int_{\Omega}\left[s_{k l}^{(u,v)}(x,y)\right]^{2} \, \mbox{d}x\,\mbox{d}y \, = \\
		& = \, (1-\alpha)\int\limits_a^b \int\limits_c^d
		    \left[ \mathbb{B}^{u,v}(x,y) \, \mathbb{S} \, cs(\mathbf{B})\right]^{\top}
		    \mathbb{B}^{u,v}(x,y) \, \mathbb{S} \, cs(\mathbf{B}) \, \mbox{d}y \,\mbox{d}x=\\
		    & = \, (1-\alpha) \left(cs(\mathbf{B})\right)^{\top} \mathbb{S}^{\top} \int\limits_a^b \int\limits_c^d
		    \left(\mathbb{B}^{u,v}(x,y)\right)^{\top} \mathbb{B}^{u,v}(x,y) \, \mbox{d}y\,\mbox{d}x \; \mathbb{S} \,  cs(\mathbf{B}).
	\end{split}
\end{equation*}
Further,
\begin{eqnarray}
\lefteqn{\int\limits_a^b \int\limits_c^d \left(\mathbb{B}^{u,v}(x,y)\right)^{\top} \mathbb{B}^{u,v}(x,y) \, \mbox{d}y\, \mbox{d}x \, =} \nonumber\\
& = &\int\limits_a^b \int\limits_c^d \left[ \mathbf{B}_{l+1-v}(y) \otimes \mathbf{B}_{k+1-u}(x) \right]^{\top} \left[ \mathbf{B}_{l+1-v}(y) \otimes \mathbf{B}_{k+1-u}(x) \right]\, \mbox{d}y\,\mbox{d}x \, = \nonumber \\
& = & \int\limits_a^b \int\limits_c^d \left[ \mathbf{B}_{l+1-v}^{\top}(y)\mathbf{B}_{l+1-v}(y)\right] \otimes \left[ \mathbf{B}_{k+1-u}^{\top}(x)\mathbf{B}_{k+1-u}(x)\right] \mbox{d}y \,\mbox{d}x \, = \nonumber \\
& = & \mathbf{M}_{l,v}^{y} \otimes \mathbf{M}_{k,u}^{x}. \nonumber
\end{eqnarray}
This yields,
$	J_2 \; = \; (1-\alpha) \left(cs(\mathbf{B})\right)^{\top} \mathbb{S}^{\top} \mathbb{M} \, \mathbb{S} \, cs(\mathbf{B}).$
By putting together the matrix forms of $J_1$ and $J_2$, the functional $J_{uv}(s_{k l}(x, y))$ from (\ref{smoothR2a}) can be expressed as a function of unknown $B$-spline parameters $b_{ij}$, specifically
\begin{equation}\label{function_J}
\begin{split}
J_{uv}\left(cs(\mathbf{B})\right) \;   =  &  \;
\alpha \left(cs(\mathbf{F})\right)^{\top} cs(\mathbf{F})- 2\alpha \left(cs(\mathbf{B})\right)^{\top} \mathbb{B}^{\top} \, cs(\mathbf{F}) + \alpha \left(cs(\mathbf{B})\right)^{\top} \mathbb{B}^{\top} \, \mathbb{B} \, cs(\mathbf{B}) + \\
& \; +(1-\alpha) \left(cs(\mathbf{B})\right)^{\top} \mathbb{S}^{\top} \mathbb{M} \, \mathbb{S} \, cs(\mathbf{B}).
\end{split}
\end{equation}
The fulfilment of the zero integral condition (\ref{cond2}) is based on relation (\ref{Ctilde}). By using this the function
$J_{uv}(cs(\mathbf{B}))$ can be reformulated as
\begin{equation}\label{JuvC}
\begin{split}
J_{uv}\left(cs(\widetilde{\mathbf{C}})\right) \,   =
& \,  \alpha \left(cs(\mathbf{F})\right)^{\top} cs(\mathbf{F}) \, - \, 2\alpha \left( \mathbb{D} \, \widetilde{\mathbb{K}} \, cs(\widetilde{\mathbf{C}}) \right)^{\top} \mathbb{B}^{\top} \, cs(\mathbf{F}) \, + \\
& + \, \alpha \left( \mathbb{D} \, \widetilde{\mathbb{K}} \, cs(\widetilde{\mathbf{C}})\right)^{\top} \mathbb{B}^{\top} \, \mathbb{B} \, \mathbb{D} \,
\widetilde{\mathbb{K}} \, cs(\widetilde{\mathbf{C}})\, +\\
& \, +   (1-\alpha) \left(\mathbb{D} \, \widetilde{\mathbb{K}} \, cs(\widetilde{\mathbf{C}})\right)^{\top}
\mathbb{S}^{\top} \mathbb{M} \, \mathbb{S} \, \mathbb{D} \, \widetilde{\mathbb{K}} \, cs(\widetilde{\mathbf{C}}).
\end{split}
\end{equation}
Thus, the necessary and sufficient condition for the minimum of function
$J_{uv}(cs(\mathbf{B}))$ is
$
\dfrac{\partial \, J_{uv}(cs(\mathbf{B}))} {\partial \, cs(\mathbf{B})} \, = \, 0.
$
By applying this condition to (\ref{JuvC}) the following equation is obtained,
$$
\widetilde{\mathbb{K}}^{\top} \, \mathbb{D}^{\top} \, \left[ (1-\alpha) \mathbb{S}^{\top} \, \mathbb{M} \, \mathbb{S} \, + \, \alpha \, \mathbb{B}^{\top} \,
\mathbb{B}\right] \mathbb{D} \, \widetilde{\mathbb{K}} \, cs(\widetilde{\mathbf{C}}) \, = \, \alpha \, \widetilde{\mathbb{K}}^{\top} \, \mathbb{D}^{\top}
\mathbb{B}^{\top} \, cs(\mathbf{F}).
$$
Then the solution to this system is given by
\begin{equation}\label{minB-thm}
cs(\widetilde{\mathbf{C^*}}) \; = \;
\left[ \widetilde{\mathbb{K}}^{\top} \mathbb{D}^{\top}
\left[(1-\alpha) \, \mathbb{S}^{\top} \mathbb{M} \, \mathbb{S} + \alpha \, \mathbb{B}^{\top} \, \mathbb{B}\right] \mathbb{D} \widetilde{\mathbb{K}}
\right]^{+} \alpha \,\widetilde{\mathbb{K}}^{\top}\mathbb{D}^{\top} \mathbb{B}^{\top} \, cs(\mathbf{F})
\end{equation}
And finally, the matrix $\mathbf{B}^*$ of coefficients for the resulting smoothing spline with zero integral is obtained by
\begin{equation}\label{finB}
cs(\mathbf{B}^*) \, = \, \mathbb{D} \, \widetilde{\mathbb{K}} \, cs(\widetilde{\mathbf{C}^*}).
\end{equation}
\end{proof}

\begin{proof}[Proof of Theorem \ref{splineR2}]
The spline $s_{kl}(x,y) \in {\cal S}_{kl}^{\Delta\lambda,\Delta\mu}(\Omega)$ can be expressed as
$$s_{kl}\left(x,y\right) \, = \, \sum\limits_{i=-k}^{g} \sum\limits_{j=-l}^{h} b_{ij} \, B_{i}^{k+1} \left( x\right)\,  B_j^{l+1} \left(y\right)
\, = \, \sum\limits_{i=-k}^{g} s_{l}^i(y) \, B_{i}^{k+1} \left( x\right),$$
where
$s_{l}^i (y) \, := \, \sum\limits_{j=-l}^{h} b_{ij} \,  B_j^{l+1} \left(y\right)$, $i=-k,\cdots,g$, are in fact one-dimensional splines
of order $l+1$ for the $y$-variable with coefficients $b_{ij}$, $j=-l,\cdots,h$. Then
$$
\int s_{kl}(x,y) \, \mbox{d}y \, = \, \int \sum\limits_{i=-k}^{g} s_l^i(y) \, B_i^{k+1}(x) \, \mbox{d}y \, = \, \sum\limits_{i=-k}^{g} B_i^{k+1}(x)
\, \int s_l^i(y) \, \mbox{d}y
$$
and
$$
\int s_l^i(y) \, \mbox{d}y \, = \, s_{l+1}^i(y), \quad \mbox{with} \quad s_{l+1}^i(y) \, = \, \sum\limits_{j=-l-1}^h u_{ij} \, B_j^{l+2}(y).
$$
By considering the case of one-dimensional splines, specifically the proof of Theorem \ref{spline}, it holds
\begin{equation}\label{vztah_u_b}
  u_{ij} \, = \, u_{i,j-1}+\dfrac{b_{ij}}{t_j}, \quad \mbox{where} \quad t_j= \dfrac{l+1}{\mu_{j+l+1}-\mu_j}, \quad j=-l,\cdots,h,
\end{equation}
i.e.
\begin{equation}\label{u}
    u_{ih} \, = \, \dfrac{b_{ih}}{t_h} + \ldots + \dfrac{b_{i,-l}}{t_{-l}}+u_{i,-l-1}, \quad \forall i.
\end{equation}
Altogether
\begin{equation*}
  \begin{split}
     \int s_{kl}(x,y) \, \mbox{d}y & = \, \sum\limits_{i=-k}^{g} B_i^{k+1}(x) \sum\limits_{j=-l-1}^{h} u_{ij} \, B_{j}^{l+2}(y) \, = \\
                                   & = \, \sum\limits_{i=-k}^{g}\sum\limits_{j=-l-1}^{h} u_{ij} \, B_{i}^{k+1}(x) \, B_{j}^{l+2}(y) \, =: \, s_{k,l+1}(x,y).
  \end{split}
\end{equation*}

Subsequently, using the last expression, the integral can be expressed as
\begin{equation}
  \begin{split}\label{i1}
     \int\limits_{c}^{d} s_{kl}(x,y) \, \mbox{d}y  & = \, \left[ s_{k,l+1}(x,y) \right]_{c}^{d} \, = \, s_{k,l+1}(x,d) - s_{k,l+1}(x,c) \, = \\
     & = \, \sum\limits_{i=-k}^{g} \sum\limits_{j=-l-1}^{h} u_{ij} \, B_{i}^{k+1}(x) \, \left( B_{j}^{l+2}(d) - B_j^{l+2}(c) \right) \, = \\
     & = \, \sum\limits_{i=-k}^{g} B_{i}^{k+1}(x) \left( u_{ih}- u_{i,-l-1} \right) \, = \, \sum\limits_{i=-k}^{g} v_i \, B_i^{k+1}(x) \, =: \, s_{k}(x),
  \end{split}
\end{equation}
for
\begin{equation}\label{def_v}
  v_{i} \, := \, u_{ih} - u_{i,-l-1} \qquad \forall i=-k,\cdots,g,
\end{equation}
because with coincident additional knots (\ref{add_knots_x1}), (\ref{add_knots_y1}) it holds
\[ B_j^{l+2}(d)= \left\{ \begin{array}{cl}
                           1 & \quad\mbox{if} \; j = h \\
                           0 & \quad\mbox{otherwise}
                         \end{array} \right. \qquad
B_j^{l+2}(c)= \left\{ \begin{array}{cl}
                        1 & \quad\mbox{if} \; j = -l-1 \\
                        0 & \quad\mbox{otherwise}.
                      \end{array} \right. \]
Finally, according to (\ref{i1}) and (\ref{int_c}), there is
\begin{equation}\label{iint}
  \begin{split}
     \int\limits_{a}^{b}\int\limits_{c}^{d} s_{kl}(x,y)\, \mbox{d}y \, \mbox{d}x \, & = \, \int\limits_{a}^{b} s_k(x) \, \mbox{d}x \, = \, \left[ s_{k+1}(x) \right]_a^b \, = \\
     & \, = s_{k+1}(b)-s_{k+1}(a) \, = \, w_g - w_{-k-1},
  \end{split}
\end{equation}
where
$s_{k+1}(x) \, = \, \sum\limits_{i=-k-1}^{g} w_i \, B_i^{k+2}(x)$ and
\begin{equation}\label{vztah_w_v}
  w_i=w_{i-1}+\dfrac{v_i}{d_i}, \quad \mbox{with}\quad d_i=\dfrac{k+1}{\lambda_{i+k+1}-\lambda_i} \quad \forall i=-k,\cdots,g,
\end{equation}
i.e.
\begin{equation}\label{w}
  w_g \, = \, \dfrac{v_g}{d_g}+\cdots +\dfrac{v_{-k}}{d_{-k}} + w_{-k-1}.
\end{equation}
As a direct consequence the following equivalences can be formulated
$$
\int\limits_{a}^{b}\int\limits_{c}^{d} s_{kl}(x,y)\, \mbox{d}y \, \mbox{d}x \, = \, 0 \quad \Leftrightarrow \quad
w_g \, = \, w_{-k-1} \quad \Leftrightarrow \quad \dfrac{v_g}{d_g}+\cdots +\dfrac{v_{-k}}{d_{-k}} \, = \, 0.
$$
By using (\ref{def_v}) and (\ref{u}),
\begin{equation*}
  \begin{split}
  \dfrac{v_g}{d_g}+\cdots +\dfrac{v_{-k}}{d_{-k}} \, & = \, \sum\limits_{i=-k}^{g} \dfrac{u_{ih}-u_{i,-l-1}}{d_i} \, = \\
  & \, = \sum\limits_{i=-k}^{g} \dfrac{1}{d_i}\left( \dfrac{b_{ih}}{t_h}+\cdots +\dfrac{b_{i,-l}}{t_{-l}} \right) \, = \,
\sum\limits_{i=-k}^{g} \, \sum\limits_{j=-l}^{h} \, \dfrac{b_{ij}}{d_i \, t_j},
\end{split}
\end{equation*}
and altogether
$$
\int\limits_{a}^{b}\int\limits_{c}^{d} s_{kl}(x,y)\, \mbox{d}y \, \mbox{d}x \, = \, 0 \quad \Leftrightarrow \quad
\sum\limits_{i=-k}^{g}\sum\limits_{j=-l}^{h} b_{ij} \left(\lambda_{i+k+1}-\lambda_i\right) \left(\mu_{j+k+1}-\mu_j\right) \, = \, 0.
$$
\end{proof}

\begin{proof}[Proof of Theorem \ref{margTh}]
Let $s_{kl}(x,y) \in {\cal S}_{kl}^{\Delta\lambda,\Delta\mu}(\Omega)$, with the given representation
$s_{kl}\left(x,y\right)=\sum\limits_{i=-k}^{g} \sum\limits_{j=-l}^{h} b_{ij} \, B_{i}^{k+1} \left( x\right)\,  B_j^{l+1} \left(y\right)$, and let
$\iint\limits_{\Omega} s_{kl}(x,y) \, \mbox{d}x \, \mbox{d}y \, = \, 0$.
Then from Theorem \ref{splineR2} it is
\begin{equation}\label{zero}
  \sum\limits_{i=-k}^{g}\sum\limits_{j=-l}^{h} b_{ij} \left(\lambda_{i+k+1}-\lambda_i\right) \left(\mu_{j+k+1}-\mu_j\right) \, = \, 0.
\end{equation}
By using (\ref{i1}), (\ref{def_v}) from the proof of Theorem \ref{splineR2} it is obtained that
$s_k(x)\, = \, \sum\limits_{i=-k}^g v_i \, B_{i}^{k+1}(x)$, where $v_i \, = \, u_{ih}-u_{i,-l-1}$.
According to (\ref{u}) it holds
\begin{equation}\label{vztah_v_b}
  v_i \, = \, \dfrac{b_{ih}}{t_h}+\cdots+\dfrac{b_{i,-l}}{t_{-l}}.
\end{equation}
Next, by considering (\ref{iint}),
$$
\int\limits_a^b s_k(x) \, \mbox{d}x \, = \, \left[s_{k+1}(x)\right]_a^b \, = \, w_g - w_{-k-1},
$$
where $s_{k+1}(x)= \sum\limits_{i=-k-1}^g w_i \, B_i^{k+2}(x)$.
But with respect to (\ref{vztah_w_v}), (\ref{w}), (\ref{vztah_v_b}) and (\ref{zero}) this difference equals to
\begin{equation*}
  \begin{split}
    w_g - w_{-k-1} & = \, \dfrac{v_g}{d_g} + \cdots + \dfrac{v_{-k}}{d_{-k}} \, = \, \sum\limits_{i=-k}^{g} \dfrac{v_i}{d_i} \, = \,
    \sum\limits_{i=-k}^{g} \dfrac{1}{d_i} \sum\limits_{j=-l}^{h} \dfrac{b_{ij}}{t_j} \, = \\
    & = \, \sum\limits_{i=-k}^{g} \sum\limits_{j=-l}^{h} \dfrac{b_{ij}\left( \lambda_{i+k+1} -\lambda_i \right)
    \left( \mu_{j+l+1}- \mu_j\right)}{(k+1)(l+1)} \, = 0,
  \end{split}
\end{equation*}
and consequently also
$\int\limits_{a}^{b} s_{k}(x) \mbox{d}x  \, = \, 0$.
The second statement can be proven analogously.
\end{proof}

\begin{proof}[Proof of Theorem \ref{sp_ind}]
Every bivariate spline $s_{kl}(x,y) \in {\cal S}_{kl}^{\Delta\lambda,\Delta\mu}(\Omega)$ can be expressed as
$$
s_{kl}(x,y) \, = \, \sum\limits_{i=-k}^g \sum\limits_{j=-l}^h b_{ij} B_{i}^{k+1}(x) B_{j}^{l+1}(y) \, = \, \sum\limits_{i=-k}^g c_{i} B_{i}^{k+1}(x),
$$
where
$c_i \, = \, \sum\limits_{j=-l}^h b_{ij} B_{j}^{l+1}(y)$.
For given univariate spline
$s_k(x)=\sum\limits_{i=-k}^g v_{i} B_{i}^{k+1}(x)$
we can define coefficients
$$v_{ij} \, := \, v_i, \; \forall j=-l,\ldots,h.$$
Then $s_k(x)$ can be expressed as a bivariate spline which is constant in variable $y$ and which uses $B$-spline bases functions $B_{j}^{l+1}(y)$ in the form
$$
s_k(x) \, = \, \sum\limits_{i=-k}^g \sum\limits_{j=-l}^h v_{ij} B_{i}^{k+1}(x) B_{j}^{l+1}(y),
$$
since with respect to the properties of $B$-splines,
\cite{deboor78,dierckx93,Schum}, we have
$$
\sum\limits_{j=-l}^h v_{ij} B_{j}^{l+1}(y) \, = \, v_i \sum\limits_{j=-l}^h
B_{j}^{l+1}(y) \, = \, v_i \cdot 1 \, = \, v_i.
$$
The rest of proof is obvious with respect to the addition or subtraction of two splines.
\end{proof}

\medskip


\section*{Supplement C: Algorithm}


Theorem \ref{splineR2} enables to formulate an algorithm for finding a bivariate tensor spline $s_{kl}(x,y) \in {\cal S}_{kl}^{\Delta\lambda,\Delta\mu}(\Omega)$ with zero integral over $\Omega$. This task is equivalent to finding the matrix $\mathbf{B} = \left(b_{ij}\right)$, $i=-k,\cdots,g$, $j=-l,\cdots,h$ of the $B$-spline coefficients:\\
{\bf 1.} choose $(g+k+1)(h+l+1)-1$ arbitrary $B$-spline coefficients $b_{ij}\in\mathbb{R}$, for $i=-k\ldots,\beta -1,\beta +1,\ldots,g$ and
 $j=-l\ldots,\gamma -1,\gamma +1,\ldots,h$,\\
{\bf 2.}  compute
$$
b_{\beta\gamma} \; = \; \dfrac{-1}{\left(\lambda_{\beta+k+1}-\lambda_\beta\right)\left(\mu_{\gamma+l+1}-\mu_\gamma \right)} \;
\sum\limits_{i=-k\atop i\neq \beta}^{g} \sum\limits_{j=-l\atop j\neq \gamma}^{h} \; b_{ij} \left(\lambda_{i+k+1}-\lambda_i\right)\left(\mu_{j+l+1}-\mu_j\right).
$$

\medskip


\section*{Supplement D: Complete set of anthropometric data}


\begin{figure}[ht]
  \centering
  \includegraphics[width=15cm]{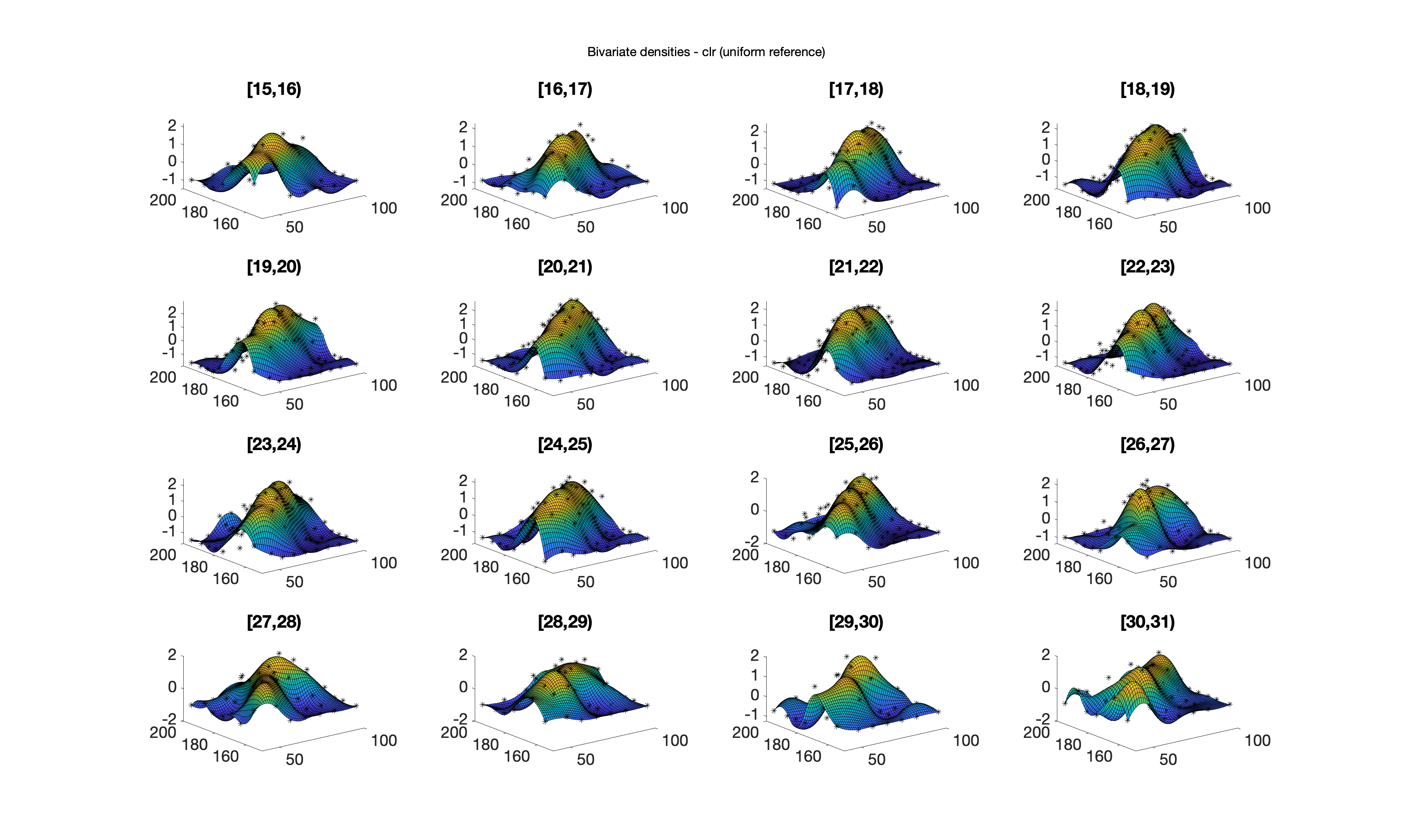}\\
  \caption{Anthropometric data: smoothed clr transformed densities for all age intervals together with data points resulting from he discrete clr transformation at mid-points of histogram classes. The choice of the scale of the reference measure (uniform measure) does not play any role here.}
\end{figure}

\begin{figure}[ht]
  \centering
  \includegraphics[width=15cm]{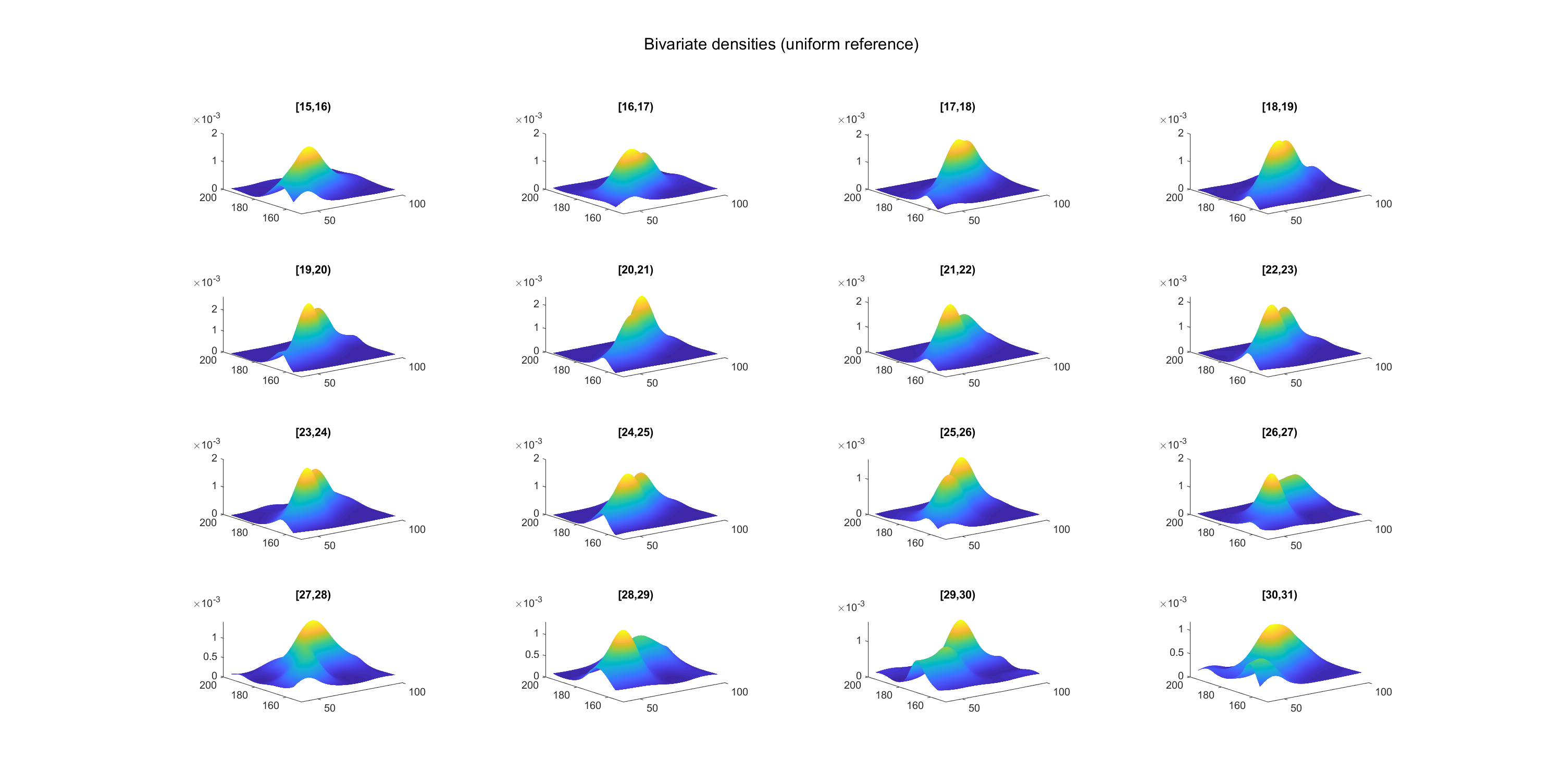}
  \caption{Anthropometric data: smoothed original bivariate densities for all age intervals.}
\end{figure}

\begin{figure}[ht]
  \centering
  \includegraphics[width=15cm]{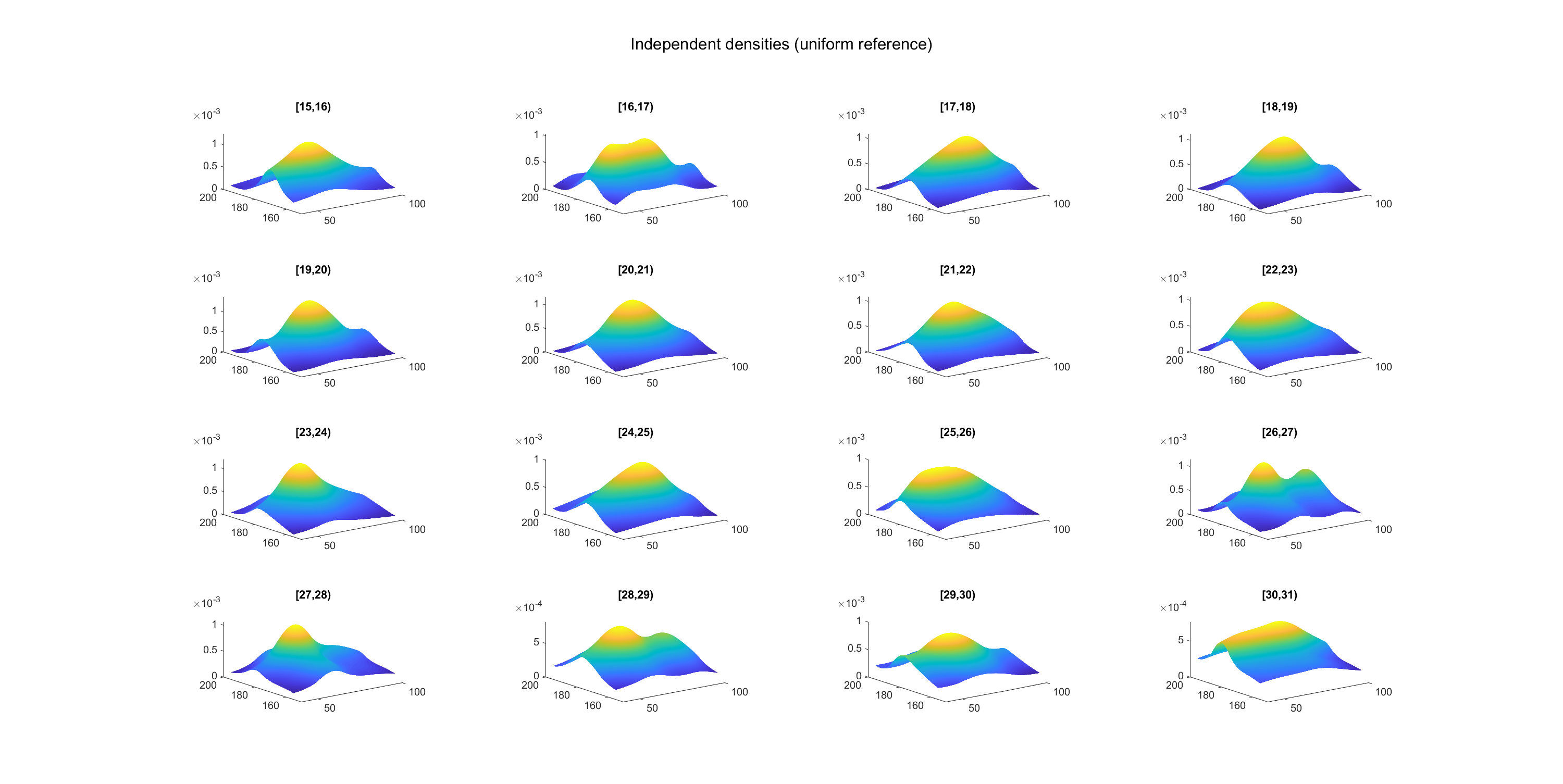}
  \caption{Anthropometric data: smoothed independent densities for all age intervals.}
\end{figure}

\begin{figure}[ht]
  \centering
  \includegraphics[width=15cm]{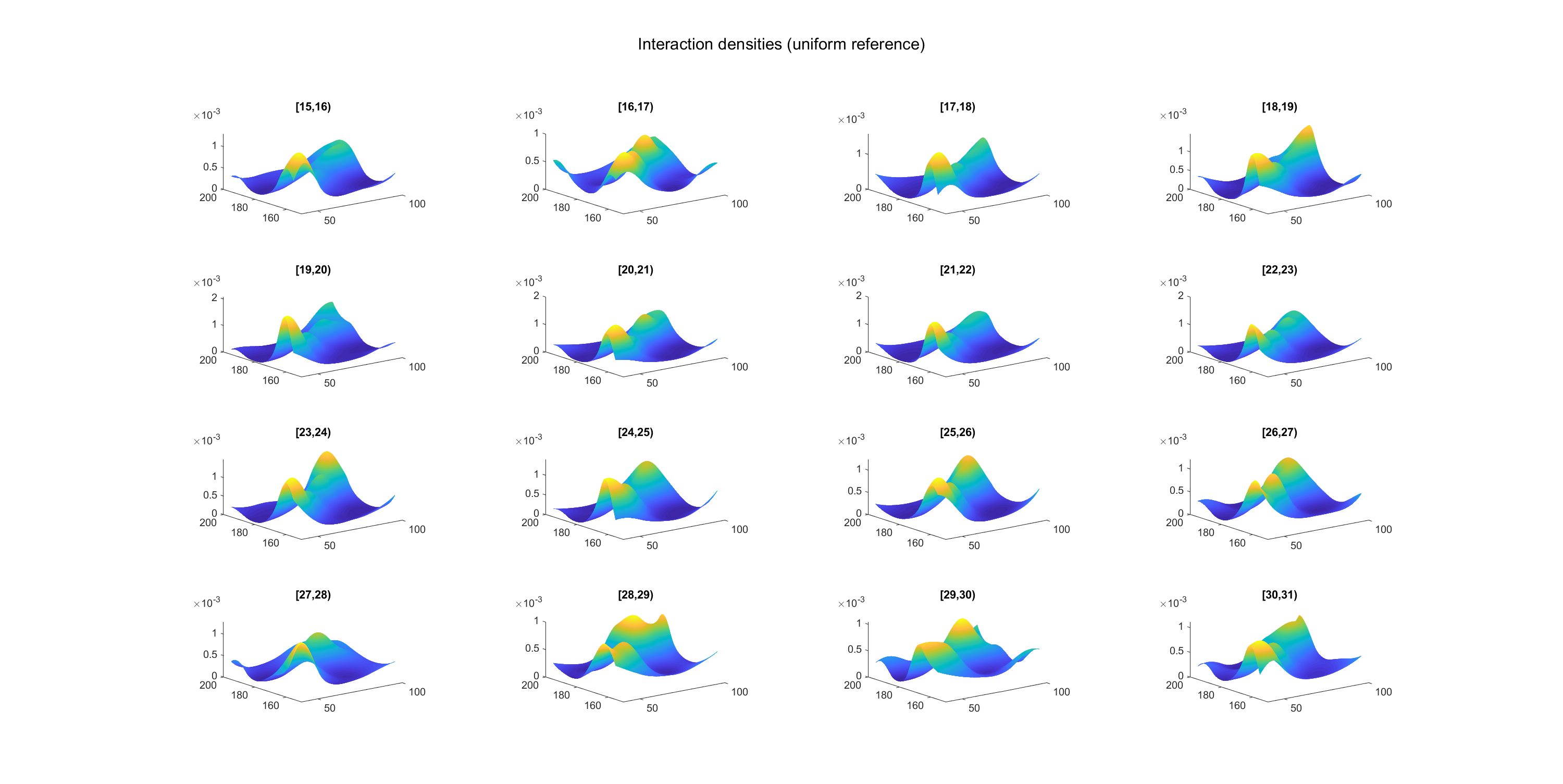}\\
  \caption{Anthropometric data: smoothed interaction densities for all age intervals.}
\end{figure}

\end{document}